\documentclass[a4paper,12pt]{article}
\usepackage[utf8]{inputenc}
\usepackage[english]{babel}
\usepackage{amsfonts}
\usepackage{latexsym,amsmath,amssymb,amsthm}

\usepackage[active]{srcltx}

\usepackage[all]{xy}
\input diagxy
\usepackage{tikz-cd}

\font\fr=eufm10 scaled \magstep 1 
\newtheorem{teor}{Theorem}
\newtheorem{prop}{Proposition}

\newtheorem{lem}{Lemma}

\newtheorem{definition}{Definition}
\theoremstyle{definition}

\theoremstyle{remark}

\newtheorem{remark}{Remark}

\def\beq{\begin{equation}}
\def\eeq{\end{equation}}
\def\bea{\begin{eqnarray}}
\def\eea{\end{eqnarray}}
\def\beann{\begin{eqnarray*}}
\def\eeann{\end{eqnarray*}}
\def\beasn{\begin{sneqnarray}}
\def\eeasn{\end{sneqnarray}}
\def\ben{\begin{enumerate}}
\def\een{\end{enumerate}}
\def\bit{\begin{itemize}}
\def\eit{\end{itemize}}

\def\qed{\ifvmode\removelastskip\fi
{\unskip\nobreak\hfil\penalty50\hbox{}\nobreak\hfil
\hbox{\vrule height1.2ex width1.2ex}\parfillskip=0pt
\finalhyphendemerits=0 \par\smallskip}}

\def\vf{\mbox{\fr X}}

\def\d{{\rm d}}

\def\Real{\mathbb{R}}

\def\Lie{\mathop{\rm L}\nolimits}

\renewcommand{\neq}{=\hspace{-3.5mm}/\hspace{2mm}}

\newcommand*{\pr}{\mathrm{pr}}
\newcommand*{\dd}{\mathrm{d}}
\newcommand*{\contr}[1]{\iota_{#1}}
\newcommand*{\Reeb}{\mathcal{R}}

\title{Optimal control, contact dynamics and Herglotz variational problem}
\author{\sffamily $^{a,b}$Manuel de Le\'on, $^a$Manuel Lainz,
$^c$Miguel C. Mu\~noz-Lecanda,
\thanks{emails: mdeleon@icmat.es, manuel.lainz@icmat.es,  
miguel.carlos.munoz@upc.edu,
}
\\[1ex]
\normalsize\itshape\sffamily 
$^a$Instituto de Ciencias Matem\'aticas(CSIC-UAM-UC3M-UCM), Madrid, Spain\\
\normalsize\itshape\sffamily 
$^b$Real Academia de Ciencias Exactas, Físicas y Naturales,  Madrid, Spain\\
\normalsize\itshape\sffamily
$^c$Department of Mathematics,
Universitat Polit\`ecnica de Catalunya,
Barcelona, Spain
}

\date{\today} 
\begin{document}
\maketitle
\begin{abstract}
 In this paper we combine two main topics in mechanics and optimal control theory: contact Hamiltonian systems and
Pontryagin Maximum Principle. As an important result, among others, we develop a contact Pontryagin Maximum Principle
that permits to deal with optimal control problems with dissipation.
We also consider the Herglotz optimal control problem, which is simultaneously a generalization of the Herglotz variational principle and an optimal control problem.
An application  to the study of a thermodynamic system is provided.  
\end{abstract}
\tableofcontents

\section{Introduction}
This paper tries to combine two important topics in mechanics and control theory: Hamiltonian contact systems and Pontryagin Maximum Principle in optimal control. 

On the one hand, Hamiltonian contact systems are getting a great popularity in recent times because they allow to describe dissipation dynamics, and several other types of physical systems in thermodynamics, quantum mechanics, circuit theory, control theory, etc. (see for instance \cite{Bravetti2019,goto2016,kho-2013,ramasb-2017,DeLeon2016b,GGMRR-2019b,silevadi-2020,Su-1999}). Recently, a generalization of contact geometry has been developed to describe field theories with dissipation \cite{GGMRR-2019, GGMRR-2020}. In fact, the Hamiltonian formulation in the scenario of contact structures exhibits very different characteristics to its counterpart in symplectic manifolds. Indeed, these differences are based on the fact that in the contact case they are Jacobi structures, more general than those of Poisson related to the simplectic ones. In variational terms, one can show that contact Hamiltonian equations can be derived from the so-called Herglotz principle, which includes as a particular case the classical Hamilton principle.

On the other hand, the Pontryagin Maximum Principle (PMP), (see \cite{Pon-1962, BaMu-2008} and references therein), is the most useful instrument for finding solutions to an optimal control problem. In fact, the PMP is the paradigm in the theory of optimal control, and since its formulation has never ceased research on its incredible properties, from very different points of view, although we will focus here on its more geometric aspects. An immediate issue arising from possible applications is that of studying problems of optimal control from the point of view of Hamiltonian contact systems, and therefore of systems with dissipative properties among many others. And, then, it seems very natural to ask whether a Pontryagin Maximum Principle could be developed to deal with a contact control problem.

Trying to look to both topics with a common viewpoint, we consider weather the solution curves to the Pontryagin Maximum Principle admit a formulation in terms of Hamiltonian contact systems in an adequate manifold and, on the contrary, if Herglotz variational problems can be understood as a particular class of optimal control problems.

With all this in mind, the paper is structured as follows. Sections 2 and 3 are dedicated to review the elements of Hamiltonian contact systems and Pontryagin Maximum Principle, both necessary to understand the object of the manuscript.

So, Section 2 is devoted just to recall the main notions and results
about contact Hamiltonian systems, including the so-called Herglotz principle, a natural extension of
the well-known Hamilton principle. As we said above, this section will faccilitate a better understanding of the rest of the paper.

Section 3 is dedicated for the Pontryagin Maximum Principle in several formulations. We introduce the classical optimal control problem, the associated extended system, the classical Pontryagin Maximum Principle and its transformation into the symplectic and presymplectic formulations. This last one is the used in several sections of the article.

In Section 4 we discuss an interesting particular case of Hamiltonian dynamics; indeed,
given a vector field $X$ on a manifold $M$, one can define the complete
lift of $X$ to its cotangente bundle $T^*M$ which is just the Hamiltonian vector field
of the Hamiltonian function determined by $X$, just its evaluation. Hence the dynamics of a general vector field is described as the corresponding to a Hamiltonian vector field in a symplectic manifold. But the dynamics on $T^*M$ 
is richer than one could expect. In fact, if the manifold $M$ decomposes as $M = \mathbb{R} \times M_o$, and the vector field has a symmetry property,
one has a very natural setting to distinguish two different  cases according to the value of the momentun $p_o$ corresponding to the global coordinate $x^o$. Indeed, one is (pre) symplectic ($p_o = 0$), and the second one, contact ($p_o \not= 0$).

Sections 5, 6 and 7 are the bulk of the paper. Section 5 is in a broader sense a direct application of Section 3.
We consider an optimal control system given by $(M, U, X, I, x_a, x_b)$ where 
$M = \mathbb{R} \times M_o$, that is, we study the so called extended system associated to an optimal control problem defined by a vector field depending on controls, $X(x,u)$, and a cost funcion $F$. Applying Theorem \ref{presym-pmp} in Section 3, we know that this problem
is equivalent to solve the dynamics of the presymplectic system $(T^* M \times U, \omega, H)$, where
$X = F \frac{\partial}{\partial x_o} + X^i \frac{\partial}{\partial x^i}$, $H = F p_o + X^i p_i$ is the linear Hamiltonian
given by $X$, and $\omega$ is the presymplectic form obtained by lifting the canonical symplectic form, $\omega_M\in\Omega^2(T^*M)$, to $T^*M \times U$. Here, $U$ represents obviously the space of controls.
The corresponding presymplectic algorithm provides the solutions, and
we can distinguish two cases: the regular one, when the controls can be obtained as functions
of the rest of variables, or the singular one, that produces higher order conditions.
Again, the evolution of the momentum $p_o$ is constant, and this permits, as above,
to discuss the cases where $p_o = 0$ or $p_o \not= 0$. With this in mind, we are able
to state the Contact Pontryagin Maximum Principle (Theorem 4).

Section 6 is just devoted to interpret the Herglotz principle as an Optimal Control Problem,
and derive the Herglotz equations of motion using the corresponding Pontryagin principle. In Section 7 we 
state the Herglotz Optimal Control Problem and find the solution equations. In this situation, the extremal condition, given as an integral of the cost function in the classical optimal control problems, is changed into an extremal condition on the solutions of a differential equation on a new variable to be maximized. This problem is a generalization of the classical optimal control systems in the sense that we obtain the classical equations if the cost function and the extremal condition is like in the classical situation. Finally, in Section 8 we apply
the above results to an example coming from Thermodynamics.

All the manifolds and mappings are considered as of $\mathcal{C}^\infty$--class. The usual Einstein convention for summation indices will be understood unless indicated. As general references for notations and basic results on geometry, mechanics and control we use \cite{abraham1978,BuLe-2005,Blo-2015}.

\section{Precontact Hamiltonian systems}

In this section we review the necessary theory of contact manifolds, contact and precontact dynamical systems, in both  Hamiltonian and Lagrangian formulations, and Herglotz variational principle and its generalized Euler-Lagrange equations. 
See \cite{arn78,bravetti2017,Bravetti2017a,DeLeon2019,GGMRR-2019,Geiges-2008,GuGoGu-96,Lainz2018,LIU2018} for details.
\subsection{Contact manifolds and Hamiltonian systems}
A \emph{contact manifold} $(M, \eta)$ is a $(2n+1)$-dimensional manifold equipped with a \emph{contact form} $\eta$, that is a 1-form satisfying
$\eta \wedge (\dd\eta)^n \not= 0$.
Then, there exist a unique vector field $\Reeb$, called the \emph{Reeb vector field},  such that
\begin{equation}
	i_{\mathcal R} \, \dd \eta = 0 \; , \qquad i_{\mathcal R}\, \eta = 1.
\end{equation}

Given $(M,\eta)$, there is a Darboux theorem for contact manifolds: around each point in $M$ one can find local \emph{Darboux coordinates} $(q^i, p_i, z)$ 
such that
\begin{equation}
	 \eta = \dd z - p_i \, \dd q^i,\quad \mathcal R = \frac{\partial}{\partial z}\, .
\end{equation}

As an example, and a natural model, we have the \emph{extended cotangent bundle} $T^*Q \times \Real$ of an $n$-dimensional manifold $Q$, which carries a natural contact form
\begin{equation}\label{tqxR,eta}
    \eta_Q = \dd z - \theta_Q,
\end{equation}
where $\theta_Q$ is the pullback of the Liouville 1-form of $T^* Q$, $ \theta_Q = p_i \dd q^i$, 
being $(q^i,p_i,z)$ the natural bundle coordinates of $T^*Q \times \Real$.

If $(M,\eta)$ is a contact manifold, the map:
\begin{align*}
    \bar{\flat} : TM &\to T^* M ,\\
     v &\mapsto \contr{v}  \dd \eta + \eta (v)  \eta.
\end{align*}
is a vector bundle isomorphism over $M$.

Given a Hamiltonian function $H:M \to \Real$, we can define a dynamical system. The triple $(M,\eta,H)$ is called a \emph{contact Hamiltonian system}. The associated Hamiltonian vector field $X_H$ is the solution to the following equation
\begin{equation}\label{Hamiltonian_vf_contact}
    \bar{\flat} (X_H) = \dd H - (\mathcal R (H) + H) \, \eta \, .
\end{equation}
In Darboux coordinates, $X_H$ has the local expression
\begin{equation}\label{eq:contact_Hamiltonian_vf_darboux}
X_H = \frac{\partial H}{\partial p_i} \frac{\partial}{\partial q^i} - 
\left({\frac{\partial H}{\partial q^i} + p_i \frac{\partial H}{\partial z}}\right)  \frac{\partial}{\partial p_i} + 
\left({p_i \frac{\partial H}{\partial p_i} - H}\right) \frac{\partial}{\partial z}\, .
\end{equation}
Therefore, an integral curve $(q^i(t), p_i(t), z(t))$ of $X_H$ satisfies the differential equations
\begin{subequations}
    \begin{align*}\label{hcont3}
    \frac{\dd q^i}{\dd t} & =  \frac{\partial H}{\partial p_i}, \\
    \frac{\dd p_i}{\dd t} & =  - \frac{\partial H}{\partial q^i} - p_i \frac{\partial H}{\partial z},\\
    \frac{\dd z}{\dd t} & =  p_i \frac{\partial H}{\partial p_i} - H.
    \end{align*} 
\end{subequations}

\subsection{Precontact manifolds and Hamiltonian systems}

Let $\eta$ be a $1$-form on an $m$-dimensional manifold $M$. We define the \emph{characteristic distribution} of $\eta$ as
\begin{equation}
  \mathcal{C}= \ker \eta \cap \ker \dd \eta \subseteq TM\, ,
\end{equation}
which we suppose to be regular. We say that $\eta$ is  a \emph{1--form of class} $c$ if the rank of the distribution $\mathcal{C}$ is $m-c$. There exist some characterizations of this notion for a 1--form given in the following~\cite{God-1969}

\begin{prop}\label{thm:class_of_form}
  Let $\eta$ be a one-form on an $m$-dimensional manifold $M$. Then, the following statements are equivalent:
  \begin{enumerate}
    \item The form $\eta$ is of class $2r+1$.
    \item At every point of $M$,
    \begin{equation}
      \eta \wedge {(\dd \eta)}^r \neq 0, \quad 
      \eta \wedge {(\dd \eta)}^{r+1} = 0.
    \end{equation}
    \item Around any point of $M$, there exist local \emph{Darboux} coordinates $x^1,\ldots x^r$, $y_1, \ldots y_r$, $z$, $u_1, \ldots u_s$, where $2r+s+1 = m$, such that
    \begin{equation}
      \eta = \dd z - \sum_{i=1}^r y_i \dd x^i.
    \end{equation}
  \end{enumerate}
\end{prop}
In these Darboux coordinates, the characteristic distribution of $\eta$ is given by
\begin{equation}
  \mathcal{C} = \langle\{ \frac{\partial }{\partial u^a} \}_{a=1,\ldots,s}\rangle.
\end{equation}

A pair $(M,\eta)$ of a manifold $M$ equipped with a form $\eta$ as above will be called a \emph{precontact manifold}, (see \cite{God-1969}). The form $\eta$ will be called a precontact form.

\begin{remark}
  The distribution $\mathcal{C}$ is involutive and it gives rise to a foliation of $M$. If the quotient $\pi:M \to M/\mathcal{C}$ has a manifold structure, then there is a unique $1$-form $\tilde\eta$ such that 
  $\pi^* \tilde{\eta}=\eta$. From a direct computation, $\tilde{\eta}$ is a contact form on $M/\mathcal{C}$. This justifies the name of \emph{precontact form}.
\end{remark}

Given $(M,\eta)$, the following map
\begin{equation}
  \begin{aligned}
    {\flat}: TM &\to TM^*\\
    v &\mapsto \contr{v} \dd \eta + \eta(v) \eta,
  \end{aligned}
\end{equation}
is a morphism of vector bundles over $M$ and its kernel is $\mathcal{C}$.

A \emph{Reeb vector field} for $(M,\eta)$ is a vector field $\Reeb$ on $M$ such that
\begin{equation}
  \contr{\Reeb} \dd \eta = 0, \qquad \eta(\Reeb) = 1,
\end{equation}
or, equivalently $\flat(\Reeb) = \eta$.

We note that there exist Reeb vector fields in every precontact manifold. Indeed we can define local vector fields $\Reeb = \frac{\partial }{\partial z} $ in Darboux coordinates and can extend it using partitions of unity. However, unlike on contact manifolds, they are not unique. In fact, given a Reeb vector field $\Reeb$ and any section $C$ of $\mathcal{C}$, we have that $\Reeb' = \Reeb + C$ is another Reeb vector field.

\subsubsection{Precontact Hamiltonian systems and the constraint algorithm}
A \textit{precontact Hamiltonian system} is a precontact manifold $(M,\eta)$ with a smooth function $H:M \to \Real$ called the Hamiltonian. We denote it by $(M,\eta, H)$.

For a precontact Hamiltonian system $(M,\eta, H)$, given a submanifold $M'\subset M$, a \textit{Hamiltonian vector field} along $M'$ is a vector field $X$, such that $X|_M\in\vf(M')$ and solution to the equation
\begin{equation}\label{eq:precontact_Hamiltonian_vf}
  \flat(X) = \dd H - (H + \Reeb(H)) \eta,
\end{equation}
at the points of $M'$, and being $\Reeb$ any Reeb vector field. It can be seen that, if this equation holds for one Reeb vector field, it will hold for all of them.

Notice that, since $\flat$ is not an isomorphism, then \eqref{eq:precontact_Hamiltonian_vf} might not have solutions at every point of the manifold $M$. Furthermore, solutions, if they exists, are not necessarily unique. Indeed, adding a section $C$ of $\mathcal{C}$ to a solution $X$ gives rise to a new solution $X' = X + C$. In order to obtain the maximal submanifold along which Hamiltonian vector fields are defined, we can develop a \emph{constraint algorithm}. To do so, let $\gamma_H=\dd H - (H + \Reeb(H)) \eta\in\Omega
^1(M)$ and define inductively $M_0 = M$, and for any positive integer $i$,
\begin{equation}
  M_{i} = \{p \in M_i \mid (\gamma_H)_p \in {\flat}(T_p M_{i-1}) \},
\end{equation}
where we assume that all $M_i$ are manifolds.

The algorithm will eventually stop, that is, we will find a positive integer $i$ such that $M_i = M_{i-1}$. We call this submanifold the final constraint submanifold $M_f$. If $M_f$ has positive dimension, there will exist Hamiltonian vector fields along $M_f$. The pair $(M_f,X)$ will be called a Hamiltonian vector field solution to the Hamiltonian precontact system $(M,\eta, H)$.

A useful characterization of such pairs is given by the following

\begin{prop}
  $X$ is a Hamiltonian vector field along $M'$ for $(M,\eta,H)$ if and only if, at the points of $M'$,
\begin{subequations}\label{eq:precontact_Hamiltonian_vf_2}
  \begin{align}
    \eta(X) &= -H,\label{eq:precontact_Hamiltonian_vf_2-1}\\
    \Lie_{X} \eta &= g \eta, \label{eq:precontact_Hamiltonian_vf_2-2}
  \end{align}
\end{subequations}
where $g:M' \to \Real$. Moreover, if this holds, then $g = - \Reeb(H)$ for any Reeb vector field $\Reeb$.
\end{prop}
\begin{proof}
  Let $X$ be a Hamiltonian vector field along $M'$. By the definition of $\flat$, equation \eqref{eq:precontact_Hamiltonian_vf}, at the points of $M'$, becomes
  \begin{equation}\label{eq1}
    \contr{X} \dd \eta + \eta(X) \eta =  \dd H - (H + \Reeb(H)) \eta,
  \end{equation}
  and, by contraction with $\Reeb$, we obtain
  \begin{equation}\label{eq2}
    \eta(X) = -H.
  \end{equation}
  Combining \eqref{eq1} and \eqref{eq2}, we deduce
  \begin{equation}
    \contr{X} \dd \eta + \dd \contr{X} \eta = - \Reeb(H) \eta,
  \end{equation}
  but the left hand side of this equation equals $\Lie_X \eta$ by Cartan's formula, hence $X$ fulfills \eqref{eq:precontact_Hamiltonian_vf_2} at the points of $M'$.

  Now assume that $X$ satisfies \eqref{eq:precontact_Hamiltonian_vf_2} on the points of $M'$. Once again, by contraction of \eqref{eq:precontact_Hamiltonian_vf_2-2} with a Reeb vector field $\Reeb$, we have
  \begin{equation}
    g = \contr{\Reeb} \Lie_X (\eta) =  \contr{\Reeb} (\contr{X} \dd \eta +  \eta(X) ) = - \contr{\Reeb}( \dd H) = -\Reeb(H).
  \end{equation}
  Combining this with \eqref{eq:precontact_Hamiltonian_vf_2}, we can easily retrieve \eqref{eq1}.
\end{proof}

\subsubsection{Morphisms of precontact Hamiltonian systems}

Let $(M,\eta, H)$ and $(\bar{M},\bar{\eta},\bar{H})$ be precontact Hamiltonian systems. A map $F:M \to \bar{M}$ is said to be a \emph{conformal morphism of precontact systems} if $F^* \bar{\eta} = f \eta$ and $F^* \bar{H} = f H$ for some non-vanishing function $f:M\to \Real$. If $f=1$, we say that $F$ is a \emph{strict morphism of precontact systems}. 
\begin{teor}\label{thm:precontact_equivalence}
  Let  $F:M \to \bar{M}$ be a conformal morphism of precontact systems. Assume that $X, \bar{X}$ are $F$-related vector fields defined along submanifolds $M' \subseteq M$ and $\bar{M}' = F(M') \subseteq \bar{M}$, respectively. Therefore, if $\bar{X}$ is a Hamiltonian vector field along $\bar M'$, then $X$ is also a Hamiltonian vector field along $M'$.
\end{teor}
\begin{proof}
  Since $\bar{X}$ is a Hamiltonian vector field, its satisfies \eqref{eq:precontact_Hamiltonian_vf_2} along $\bar{M}'$
  \begin{subequations}
    \begin{align}
      \bar{\eta}(\bar{X}) &= -\bar{H},\\
      \Lie_{\bar{X}} \bar{\eta} &= \bar{g} \bar{\eta},
    \end{align}
  \end{subequations}
  
  Pulling back by $F$, we obtain
  \begin{subequations}
    \begin{align}
      f \eta(X) = -f H ,\\
      \Lie_{{X}} (f {\eta}) &= (\bar{g}\circ F) f {\eta}.
    \end{align}
  \end{subequations}
  Reorganizing, we obtain
  \begin{subequations}
    \begin{align}
       \eta(X) = - H ,\\
      \Lie_{X} ( {\eta}) &= g {\eta},
    \end{align}
  \end{subequations}
  where $g = \bar{g}\circ F - (\Lie_{X} f)/f$. Hence $X$ is a Hamiltonian vector field.
\end{proof}

Observe that if $F$ is a diffeomorphism, then we have a bijective correspondence between pairs of Hamiltonian vector fields along submanifolds.

\subsection{The Lagrangian formalism}
Unlike $T^*Q \times \Real$, the manifold $TQ \times \Real$ does not have a canonical contact structure. However, given a \emph{Lagrangian function} $L:TQ \times \Real \to \Real$ one can construct the $1$-form
\begin{equation}\label{eq:Lagrangian_form}
    \eta_L = \dd z - \theta_L,    
\end{equation}
where $\theta_L$ is the associated Lagrangian 1-form, which in bundle coordinates $(q^i, v^i, z)$ is written as
\begin{equation}
    \theta_L = \frac{\partial L}{\partial v^i} \dd q^i.
\end{equation}

The Lagrangian $L$ is said to be \emph{regular} if its Hessian matrix with respect to the velocities,
\begin{equation}
    (W_{ij}) = \left( \frac{\partial^2 L}{\partial v^i \partial v^j} \right),
\end{equation}
is regular.

One can see that $\eta_L$ is contact form when $L$ is regular. Furthermore, $\eta_L$ is a precontact form when $(W_{ij})$ has constant rank (see~\cite[Section]{DeLeon2019}).

The energy of the Lagrangian is $E_L = \Delta (L) - L$
where $\Delta$ is the canonical Liouville vector field on $TQ$, $ \Delta = v^i \frac{\partial}{\partial v^i}$, extended in the usual way to $TQ \times \Real$ with the same local expression.

Hence, provided $L$ is such that $(W_ij)$ has full (resp. constant) rank we have that $(TQ \times \Real, \eta_L, E_L)$ is a contact (resp. precontact) Hamiltonian system. Let $\xi_L$ be a Hamiltonian vector field for this contact or precontact system. From a direct computation one can see that every integral curve $(q^i(t), v^i(t), z(t))$ of $\xi_L$ is a solution of the \emph{Herglotz equations}:
\begin{equation}\label{eq:Herglotz}
\frac{\dd}{\dd t} \left(\frac{\partial L}{\partial v^i}\right) - \frac{\partial L}{\partial q^i} =
\frac{\partial L}{\partial v^i} \frac{\partial L}{\partial z},
\end{equation}
and $\dot{z}(t) = L(q^i(t), v^i(t), z(t))$. These equations are also called generalized Euler--Lagrange equations.

Notice that, in the contact case, $\bar{\xi}_L$ is a second order differential equation, a SODE, meaning that its integral curves satisfy $v^i(t) = \dot{q}^i(t)$ . In the precontact case, the situation is more subtle. If there exist solutions, which are not necessarily unique, there is at least one which is a SODE. The details are explained in~\cite[Section~10]{DeLeon2019}.

\subsubsection{The Herglotz variational principle}\label{Hvp}
The integral curves of a contact Lagrangian system can also be obtained from a variational principle. Unlike in the case of Hamilton's principle, the action is not an integral of the Lagrangian, but it is given by an ordinary differential equation on a new variable $z$.

Given a Lagrangian function, $L:TQ \times \Real \to \Real$, for $q_o,q_1\in Q$, we consider the set $\Omega(q_0,q_1)$ of curves $\gamma:[a,b] \to \Real$ such that $\gamma(a) = q_0$,  $\gamma(b) = q_1$; and fix $z_0 \in \Real$. We define the functional
\begin{equation}
    \mathcal{Z}:\Omega(q_0,q_1) \to \mathcal{C}^\infty ([a,b] \to \mathbb{R}),
\end{equation}
 which assigns to each curve $\gamma$ the curve $\mathcal{Z}(\gamma)$ that solves the following ODE:
\begin{equation}\label{contact_var_ode}
\begin{aligned}
    \frac{\dd\mathcal{Z}(c)}{\dd t} &= L(c, \dot{c}, \mathcal{Z}(c)),\\
    \mathcal{Z}(\gamma)(a) &= z_0.
    \end{aligned}
\end{equation}
Finally, the action is given by evaluating the solution at the endpoints:
\begin{equation}
\begin{aligned}
    \mathcal{A}:\Omega(q_0,q_1) \to \Real,
    \gamma &\mapsto \mathcal{Z}(\gamma)(b)
\end{aligned}
\end{equation}
Using techniques from calculus of variations\cite[Section~5]{DeLeon2019}, one can proof the following:
\begin{teor}[Contact variational principle]
    Let $L: TQ \times \Real \to \Real$ be a Lagrangian function and let $\gamma \in  \Omega(q_o,q_1)$. Then, $(\gamma,\dot\gamma, \mathcal{Z}(\gamma))$ satisfies the Herglotz's equations \eqref{eq:Herglotz} if and only if $\gamma$ is a critical point of $\mathcal{A}$.
\end{teor}
These Herglotz equations, called also \emph{generalized Euler--Lagrange equations}, are
$$
 \frac{\d}{\d t}\left(\frac{\partial L}{\partial v^i}\right)_\gamma-\frac{\partial L}{\partial q^i}
-\frac{\partial L}{\partial z}\frac{\partial L}{\partial v^i}=0\, .
 $$
 Observe that they are not linear on the Lagrangian.
 
In Section \ref{Hvp-ocp} we provide a new proof of this last statement based on the Pontryagin Maximum Principle.

\section{A quick survey on optimal control and Pontryagin
Maximum Principle}
Roughly speaking, for our interest the Pontryagin Maximum Principle, PMP, transforms an optimal control problem into a presymplectic one. The method is to mimic the lifting of a vector field, $X\in\mathfrak{X}(M)$, to the cotangent bundle, $T^*M$, using the Hamiltonian function associated to the natural operation, by duality, of the vector field $X$ on the cotangent bundle. This is done for a control depending vector field but in the particular case where the original manifold $M$ is the product $M=\Real\times M_o$ where $M_o$ is a manifold. 

This Section tries to introduce what is an optimal control problem and how works the Pontryagin Maximum Principle with the adequate approach  for our interest. For a clearest exposition,  we suppose that all the manifolds and mappings are of $\mathcal{C}^{\infty}$-class.

Since the original result and proof of Pontryagin and collaborators, \cite{Pon-1962}, there are numerous expositions with applications and proofs on the Pontryagin principle; in this review we follow \cite{BaMu-2008} for notations and statements. There a detailed proof is given and a extensive bibliography is included.

\subsection{The optimal control problem}

\subsubsection{Statement of the problem}

Consider the following diagram:

$$
\xymatrix{&TM_o\ar[d]^{\tau_o}\\{M_o\times U}\ar[ur]^{X_o}\ar[r]^{\quad\pi_1}&M_o}
$$

whith the following elements:
\begin{enumerate}
\item $M_o$ is a differentiable manifold, $\dim M_o=m_o$, the \textit{state space} for the vector field $X_o$. The points in $M_o$ will be denoted by $x$ and, when necessary, the coordinates in $M_0$ will be denoted by $(x^i)$.
\item $U\subset\Real^k$ is called the \textit{control} set. Its elements are denoted by $u$, the controls, and we denote by $(u^a)$ its local coordinates, that is $u=(u^1,\ldots,u^k)$.
\item $X_o$ is a \textit{vector field along the projection} $M_o\times U\to M_o$. Given $u\in U$ we denote by $X_o^u=X_o(\, .\,,u)\in\mathfrak{X}(M_o)$. It gives the dynamics of the problem.
\end{enumerate}

Suppose that we have given a function $F:M_o\times U\to\Real$, an interval $I=[a,b]\subset\Real$ and $x_a,x_b\in M_o$. With all this elements $(M_o,U,X_o,F,I,x_a,x_b)$ we have the following 

\bigskip
\noindent\textbf{Optimal control problem}, OCP: Find curves $\gamma:I\to M_o\times U$, $\gamma=(\gamma_o,\gamma_U)$, such that
\begin{enumerate}
  \item[1)] end points conditions: $\gamma_o(a)=x_a, \gamma_o(b)=x_b$,
  \item[2)] $\gamma$ is an integral curve of $X_o$: $\dot{\gamma}_o=X_o\circ\gamma$\, , and
  \item[3)] minimal condition: $S[\gamma]=\int_a^b F(\gamma(t))\d\,t$ is minimum over all curves satisfying 1) and 2). 
  \end{enumerate}

The function $F$ is called the \textit{cost function} of the problem.

\medskip
In local coordinates, if $X =X^i\frac{\partial}{\partial x^i}$, then the differential equation for the curve $\gamma$ are
$$
\dot{x}^i =X^i(x^j,u)\, .
$$
The minimal condition allows to obtain the solution for the controls $u=u(t)$. Introducing them in the differential equation and integrating them we have the curves solution of the optimal control problem.  
\subsubsection{The extended optimal control problem}

To solve the above problem it is necessary to incorporate into the vector field the cost function as a direction in the tangent bundle of the state space. This is made by the construction of the so called \textit{extended problem}.

Associated with the previous elements, consider the diagram:

$$
\xymatrix{&TM=T\Real\times TM_o\ar[d]^{\tau}\\{M\times U=\Real\times M_o\times U}\ar[ur]^{X_o}\ar[r]^{\qquad\pi_1}&M=\Real\times M_o}
$$

\noindent where the points in $M=\Real\times M_o$ are denoted by $(x^o,x)$, and the vector field $X$ along the projection $\pi_1$ is
$$
X=F\frac{\partial}{\partial x^o}+X_o\, .
$$

\begin{remark}
 Observe that  $[\partial/\partial x^o,X]=0$, hence we are in a  situation where the direction associated to $x^o$ is specifically identified. In particular this implies that the vector field $X$ is projectable to $M_o$. This situation is going to be used in other parts of this and other sections.
\end{remark}

From the original elements we have at the beginning, $(M_o,U,X_o,F,I,x_a,x_b)$, we now have $(M,U,X,I,x_a,x_b)$ and we consider the following problem:

\bigskip
\noindent\textbf{Extended optimal control problem}, EOCP: 
Find curves $\hat\gamma:I\to \Real\times M_o\times U$, $\hat\gamma=(\gamma^o, \gamma_o,\gamma_U)$, such that
\vspace{-2mm}\begin{enumerate}
  \item[1)] end points conditions: $\gamma_o(a)=x_a, \gamma_o(b)=x_b, \gamma^o(a)=0$,
  \item[2)] $\hat\gamma$ is an integral curve of $X$: $\dot{\overline{(\gamma^o,\gamma_o)}}=X\circ\hat\gamma$\, , and
  \item[3)] maximal condition: 
  $x^o(b)$ is maximal over all curves satisfying 1) and 2). 
  \end{enumerate}
  
  Remember that $F$ is the cost function of the original optimal control problem.
  
  This extended optimal control problem is equivalent to the initial optimal control problem as defined above, that is there is a bijection between the set of solutions $\gamma$ of the first problem and the set of solution $\hat\gamma$ of the second one corresponding to the variables $x^1,\ldots,x^{m_o}$. The variable $x^o$ is not relevant to the problem, it is and additional variable used to identify the direction with maximal increment in the tangent bundle to $M$ and to prove the Pontryagin Maximum principle.

In the sequel we only consider this form of the optimal control problem
 and we always refer to this statement as optimal control problem. We denote it by $(M,U,X,I,x_a,x_b)$.

\subsection{The Pontryagin Maximum Principle}

As we have said above, the solution to this problem was obtained by Pontryagin and collaborators in 1954. For a modern proof and applications, see \cite{BaMu-2008} and references therein.

Given the above optimal control problem $(M,U,X,I,x_a,x_b)$, for any $u\in U$, we consider the symplectic problem given by
\begin{enumerate}
  \item Manifold: $T^*M$.
  \item Symplectic form $\omega_M$, the 2-canonical form of $T^*M$.
  \item Hamiltonian function: $H^u=\hat X^u=p_oF^u+p_i(X^u)^i$.
\end{enumerate}

Where we have denoted by $X^u$ the vector field $X(\, .\, , u)$, and similarly with the other elements. The Hamiltonian function is the natural one associated to the vector field $X^u$ on the cotangent bundle $T^*M$.
We call this problem $(T^*M,\omega_M,H^u)$. It is Hamiltonian symplectic system.

As we know, the associated Hamiltonian vector field, $X_H^u$, defined by $\mathbf{i}(X_H^u)\omega_M=\d\,H^u$, is locally given by
\begin{equation} X_H^u=F^u\frac{\partial}{\partial x^o}+(X^u)^i\frac{\partial}{\partial x^i}-
\left(\lambda_o\frac{\partial F^u}{\partial x^i}+
p_j\frac{\partial (X^u)^j}{\partial x^i}\right)\frac{\partial}{\partial p_i}\, .
\end{equation} 

All this no more than the canonical lifting of a vector field $X$ on a manifold $M$ to its cotangent bundle $T^*M$ and denoted usually by $X^*$, in this particular case $(X^u)^*$. We will go on this ideas on the following section with more detail and other points of view.

With this in mind we have: (see \cite{BaMu-2008} for a detailed proof)

\begin{teor}: \textbf {Pontryagin Maximum Principle}\label{pmp}

Given the optimal control problem $(M,U,X,I,x_a,x_b)$, let 
$\hat\gamma:I\to \Real\times M_o\times U$ be a solution, $\hat\gamma=(\gamma_{M},\gamma_U)$,  then there exists $\hat\sigma:I\to T^*M\times U=T^*\Real\times T^*M_o\times U$, $\hat\sigma=(\sigma_{T^*M},\sigma_U)$ such that
\begin{enumerate}
  \item[1)] it is a solution to the Hamiltonian problem $(T^*M\times U,\omega,H^u)$, that is it is an integral curve of $X_H^u$, for some fixed $u\in U$,
  \item[2)] $\hat\gamma=\pi\circ\hat\sigma$, where $\pi:T^*M\times U\to M\times U$ is the natural projection, and $\hat\gamma$ satisfies the end points condition; hence $\sigma_U=\gamma_U$, 
  \item[3)] $H(\sigma_{T^*M}(t), \gamma_U(t))=\mathrm{sup}_{v(t)\in U}H(\sigma_{T^*M}(t), v(t))$ for every $t\in I$.
\end{enumerate}	
\end{teor}

This Theorem gives a necessary condition the solutions must fulfill. The way it is applied is as follows: condition 3) allows to obtain the solution for $u(t)$ and with this solution we can integrate the Hamiltonian vector field $X_H^u$, obtaining the curves $\hat\sigma(t)$ and hence $\hat\gamma(t)$ and the initially desired solution $\gamma_o(t)$.

The differential equations defining the integral curves of $X_H^u$ are the following:

\begin{equation}\label{pmp-eqs}
\begin{matrix}
\dot x^o=\frac{\partial H^u}{\partial p_o}=F\,\, , & \dot p_o=
\frac{\partial H^u}{\partial x^o}=0,\,\,(\Rightarrow p_o=ct) 
\\
& &
\\
\dot x^i=\frac{\partial H^u}{\partial p_i}=X^i\,\, , &  
\dot p^i=-\frac{\partial H^u}{\partial x^i}=-p_o\frac{\partial F}{\partial x^i}-p_j\frac{\partial X^j}{\partial x^i}
\end{matrix}
\end{equation}

\bigskip
As we are assuming that all the elements of the problem are of $\mathcal{C}^{\infty}$-class,  and we suppose furthermore that $U\subset\Real^k$ is an open set, then  condition 3) in the Theorem can be changed to

\bigskip
3') $\frac{\partial H}{\partial u}|_{\hat\sigma(t)}=0$ for every $u\in U$.

\medskip
Hence in order to obtain the solution $\gamma_U$, if possible, we have this last expression as other equations to add to (\ref{pmp-eqs}).  If $(u^1,\ldots,u^k)$ is a basis for $\Real^k$,  we have the equations

\begin{equation}\label{pmp-controls}
\frac{\partial H}{\partial u^1}=0,\ldots,\frac{\partial H}{\partial u^k}=0
\end{equation} 
together with equations (\ref{pmp-eqs}) to solve the optimal control problem.

In the sequel we will assume that $U$ is an open subset of $\Real^k$.

Then instead of Theorem \ref{pmp}, we have the following

\begin{teor}: \textbf {Weak Pontryagin Maximum Principle}\label{weak-pmp}

Given the optimal control problem $(M,U,X,I,x_a,x_b)$, with $U\subset\Real^k$ an open set, let 
$\hat\gamma:I\to \Real\times M_o\times U$ be a solution, $\hat\gamma=(\gamma_{M},\gamma_U)$,  then there exists $\hat\sigma:I\to T^*M\times U=T^*\Real\times T^*M_o\times U$, $\hat\sigma=(\sigma_{T^*M},\sigma_U)$ such that
\begin{enumerate}
  \item[1)] it is a solution to the Hamiltonian problem $(T^*M\times U,\omega,H^u)$, that is, it is an integral curve of $X_H^u$, for any fixed $u\in U$,
  \item[2)] $\hat\gamma=\pi\circ\hat\sigma$, where $\pi:T^*M\times U\to M\times U$ is the natural projection, and $\hat\gamma$ satisfies the end points condition; hence $\sigma_U=\gamma_U$, 
  \item[3)] minimality conditions: $\frac{\partial H}{\partial u}|_{\hat\sigma(t)}=0$ for every $u\in U$ and for every $t\in I$.
\end{enumerate}	
\end{teor}  

\subsection{The presymplectic approach to PMP}

Now we try to give another approach to the Pontryagin Maximum Principle more adequate for our problems. It is stated as a presymplectic problem  and goes as follows.

Consider the problem given by $(M,U,X,I,x_a,x_b)$ and the solution by means of the symplectic system $(T^*M,\omega_M,H^u)$ with equations (\ref{pmp-eqs}) and (\ref{pmp-controls}). Take the projection
$$
\pi_1:T^*M\times U\to T^*M
$$

\noindent and the 2-form $\omega=\pi_1^*\,\omega_M\in\Omega^2(T^*M\times U)$. It is a presymplectic form and its kernel is given by
$$
\ker\omega=\left\{\frac{\partial }{\partial u^1},\ldots,\frac{\partial }{\partial u^k}\right\}\, .
$$

We can consider the presymplectic system $(T^*M\times U,\omega, H)$ whose dynamical equation is given by 
$$\mathbf{i}(X_H)\,\omega=\d\, H\, .$$

Being a presymplectic system, the compatibility equations are given by $\mathbf{i}(Z)\d\,H=0$ for every $Z\in\ker\omega$, that is equations (\ref{pmp-controls}).

Changing Theorem \ref{weak-pmp} to this new situation we have

\begin{teor}: \textbf {Presymplectic Pontryagin Maximum Principle}\label{presym-pmp}

Given the optimal control problem $(M,U,X,I,x_a,x_b)$, with $U\subset\Real^k$ an open set, let 
$\hat\gamma:I\to M\times U=\Real\times M_o\times U$ be a solution, 
$\hat\gamma=(\gamma_M,\gamma_U)$,  then there exists $\hat\sigma:I\to T^*M\times U=T^*\Real\times T^*M_o\times U$, $\hat\sigma=(\sigma_{T^*M},\sigma_U)$ such that
\begin{enumerate}
  \item[1)] it is a solution to the Hamiltonian presymplectic problem $(T^*M\times U,\omega,H)$, that is it is an integral curve of $X_H$, solution to the equation $\mathbf{i}(X_H)\,\omega=\d\, H$,
  \item[2)] $\hat\gamma=\pi\circ\hat\sigma$, where $\pi:T^*M\times U\to M\times U$ is the natural projection, and $\hat\gamma$ satisfies the end points condition; hence $\sigma_U=\gamma_U$, 
  \item[3)] minimality, compatibility, conditions: $\frac{\partial H}{\partial u}|_{\hat\sigma(t)}=0$ for every $u\in U$ and for every $t\in I$.
\end{enumerate}	
\end{teor}

A solution to the equation $\mathbf{i}(X_H)\,\omega=\d\, H$ is given by:
\begin{equation} X_H=F\frac{\partial}{\partial x^o}+X^i\frac{\partial}{\partial x^i}-
\left(\lambda_o\frac{\partial F}{\partial x^i}+
p_j\frac{\partial X^j}{\partial x^i}\right)\frac{\partial}{\partial p_i}\, .
\end{equation}

Observe that this solution exists all over the manifold
$T^*M\times U$ and that $p_o$ is constant for every curve solution to the problem.

Suppose that the compatibility equations allow us to determine the controls $u^1,\ldots,u^k$, that is we can obtain $u^a=	\psi(x^o,x^i,p_o,p_i)$, then we say that the optimal control problem is \textbf{regular}, otherwise it is called \textbf{singular}. In the singular case, it is necessary to apply an algorithm of constraints, that is to go to higher order conditions, to obtain the controls perhaps on a submanifold of $T^*M\times U$. See \cite{BaMu-2008,BaMu-2012} for details on these ideas and \cite{got79} for the used algorithm.

Note that the weak and the presymplectic approaches to the maximum principle are equivalent since the local equations are the same.

\begin{remark}
 Along this appendix and for simplicity in the exposition, we have considered that the set of controls $U$ is an open set in an Euclidean space, hence we have the product $M\times U$. We can change this situation by a non trivial bundle $C\to M$, instead of the natural projection $M\times U\to M$, considering the controls as the elements of the fibres. The local equations are the same that we have obtained in the trivial case for the controls. 
\end{remark}

\section{Dynamics of vector fields as contact dynamics}\label{vf-codyn}

It is well known that the integral curves of a vector field in a manifold $M$ can be obtained as projection of integral curves of a Hamiltonian vector field in the cotangent bundle. We can extend this dynamics to the contact associated manifold $TM\times\Real$, as in equation (\ref{tqxR,eta}), what gives the additional equation $\dot z=0$, that is in a trivial way. We want to obtain a non trivial extension.

In this section we study how to obtain these integral curves as solutions of a contact dynamical system in an adequate contact manifold, at least in the case that the original vector field has some symmetry properties. Here we recover a similar situation we had in the Pontryagin Maximum Principle in its symplectic approach. See Section 3.

\subsection{The general case}
Let $M$ be a manifold and $X\in \mathfrak{X}(M)$ a vector field. Let $\hat X:T^*M\to\Real$ the natural function defined by $\hat X(\alpha)=\alpha(X)=<\alpha,X>$. In a canonical coordinate system $(x^i,p_i)$ in $T^*M$, we have that $\hat X(x,p)=p_iX^i$.

As it is well known, if $\omega_M=-\d\,\theta_M$ is the  symplectic canonical 2-form in $T^*M$, we can consider the Hamiltonian symplectic system $(T^*M,\omega_M,\hat X)$. Then the Hamiltonian vector field $Y_{\hat X}\in\mathfrak{X}(T^*M)$, defined by $\mathbf i(Y_{\hat X})\omega_M=\d\,\hat{X}$, has local expression
$$
X=X^i\frac{\partial}{\partial x^i},\,\,\,\Rightarrow\,\,\,   Y_{\hat X}=X^i\frac{\partial}{\partial x^i}-p_j\frac{\partial X^j}{\partial x^i}\frac{\partial}{\partial p_i}
$$  
if $(x^i)$ and $(x^i,p_i)$ are coordinates of $M$ and $T^*M$ respectively. By this local expression we have that $Y_{\hat X}=X^*$, where $X^*$ is the so called \textit{canonical lifting} of $X\in\vf(M)$ to $T^*M$. The integral curves of $Y_{\hat X}$ projected to $M$ are the integral curves of $X$ as we can see by direct observation of the above local expression. With this method, we have transformed any vector field in a Hamiltonian one but doubling the dimension. For details about these constructions we refer to
\cite{deLeon-Rod-1989,YaIs-1973}.

Observe that the Hamiltonian $\hat X$ depends linearly on the momenta.

\subsection{The case 
 \texorpdfstring{$M=\Real\times M_o$}{}}

\subsubsection{The symplectic case}\label{sympl-ext}

Suppose now that we have one direction specially identified in the tangent bundle to the manifold, that is  $M=\Real\times M_o$. When necessary we denote by $(x^o,x^i)$ a coordinate system in $M$ and $(x^o,x^i,p_o,p_i)$ its natural extension to $T^*M$.

Let $X\in \mathfrak{X}(M)$ and suppose that 
$$
\left[\frac{\partial}{\partial x^o},X\right]=0\,.
$$
In coordinates this means that, if $X=X^o\frac{\partial}{\partial x^o}+X^i\frac{\partial}{\partial x^i}$, then the coordinates $X^o$ and $X^i$ of the vector field $X$ do not depend on $x^o$. In particular this implies that $X$ is projectable to $M_o$.

\bigskip
\begin{remark}
 What is the meaning of this situation? Suppose we have two vector fields $X_o,X\in\vf(M)$ with $[X_o,X]=0$. Then around any regular point of $X_o$ we can choose a local coordinate system $(U,x^o,x^i)$, with $i=1,\ldots,n$, if $\dim M=1+n$, and $U\subset M$ an open set, with $X_o|_U=\partial/\partial x^o$. Hence we have the above situation but locally. In this case the local decomposition $\{x^o\}\times\{x
^i\}$ is not unique.

This is what we called above ``\textit{particular symmetry property}" for the vector field $X$. We can observe that it is a common situation at least locally.
\end{remark}

\bigskip

This is a situation we are going to tackle when trying to relate contact structures and optimal control. The variable $x^o$ will correspond to the cost function $F$ as we have seen in Section 3 in our review of the Pontryagin Maximum Principle.

If we proceed in this case as above in the general situation, with $\mathbf i(X^*)\omega_M=\d\, H$, where the Hamiltonian function $H$ is defined by
$$
H=\hat{X}= p_oX^o+p_iX^i
$$
then the corresponding Hamiltonian vector, using $
[\partial/\partial x^o,X]=0
$, is given by 
$$
X^*=X^o\frac{\partial}{\partial x^o}+X^i\frac{\partial}{\partial x^i}-0\frac{\partial}{\partial p_o}-\left(p_o\frac{\partial X^o}{\partial x^i}+
p_j\frac{\partial X^j}{\partial x^i}\right)\frac{\partial}{\partial p_i}\, .
$$

The associated system of differential equations is:
$$
\dot x^o=X^o,\,\,\dot p_o=0,\,\, \dot x^i=X^i,\,\, \dot p^i=-p_o\frac{\partial X^o}{\partial x^i}-p_j\frac{\partial X^j}{\partial x^i}
$$
This is the description of the Hamiltonian system $(T^*M,\omega_M,H)$ with $H=\hat{X}$.

\subsubsection{The relation with contact dynamics}\label{normalsolut}

Observe that the vector field $X^*$ is tangent to the submanifold defined by $p_o=\mathrm{constant}$, hence we can reduce the problem to those hypersurfaces of $T^*M$. We have two different situations and, by comparison with the situation of the optimal control and the symplectic Pontryagin Maximum Principle, we will call \textbf{normal} and \textbf{abnormal} situations.

\bigskip
\noindent\textbf{a) The normal situation: $p_o\neq 0$}

\medskip
For $\lambda_o\in\Real$, $\lambda_o\neq 0$, let $N\subset T^*M$ be the submanifold defined by $p_o=\lambda_o$ and let $j:N\hookrightarrow T^*M$ be the natural inclusion. Obviously the dimension of $N$ is odd, hence it can not be a symplectic manifold. We denote by $(x^o,x^i,p_i)$ the coordinates induced in $N$ by the coordinates we have in $T^*M$.

Consider now the canonical 1-form $\theta_M\in\Omega^1(T^*M)$ and let $
\eta=-j^*\theta_M$, then we have the following result

\begin{lem}$(N,\eta)$ is a contact manifold. The Reeb vector field is $R=-\frac{1}{\lambda_o}\frac{\partial}{\partial x^o}$.
	
\end{lem}

 The proof is direct using its local expression, $\eta=-\lambda_o\d\,x^o-p_i\d\, x^i$. The minus sign comes from a convention in the definition of the symplectic form in $T^*M$ and the 1-form and 2-form in a contact manifold.
 
 Let $H_N=j^*H$ be the restriction of $H$ to $N$. We have that , locally, $H_N= \lambda_oX^o+p_iX^i$ and we have a \textbf{Hamiltonian contact system} given by $(N,\eta,H_N)$. Let $X_N\in\mathfrak{X}(N)$ be the corresponding contact Hamiltonian vector field, that is:
 $$
 \mathbf i(X_N)\eta=-H_N,\,\,\,\mathbf i(X_N)\d\,\eta=\d\, H_N-(L(R)H_N)\eta\, 
 $$
 whose local expression is
\begin{equation}\label{norm-cont}
 X_N=X^o\frac{\partial}{\partial x^o}+X^i\frac{\partial}{\partial x^i}-
\left(\lambda_o\frac{\partial X^o}{\partial x^i}+
p_j\frac{\partial X^j}{\partial x^i}\right)\frac{\partial}{\partial p_i}\, ,
\end{equation}
with the usual notation confusing the functions on $T^*M$ and their restrictions to $N$.

With this in mind, we have that:

\begin{teor}\

 The vector field $X^*\in\mathfrak{X}(T^*M)$ is tangent to $N$ and, on the points of $N$, it is equal to $X_N$.
	
 \end{teor}
 
 Hence the normal integral curves to the vector field $X^*$ are solutions of a Hamiltonian contact dynamics on a corresponding contact manifold. The contact system is $(N,\eta, H_N)$.
 
\bigskip

 	\noindent\textbf{Comment: A little calculus} 
 
 Here we give the corresponding calculus to obtain the expression in (\ref{norm-cont}).
 
 We have that $H_N= \lambda_oX^o+p_iX^i$ and $\eta=-\lambda_o\d\,x^o-p_i\d\, x^i$. Denoting $X_N$ by
 $$ 
 X_N=a^o\frac{\partial}{\partial x^o}+a^i\frac{\partial}{\partial x^i}+b_i\frac{\partial}{\partial p_i}
$$ 
the first contact dynamical equation is:
$$
\mathbf i(X_N)\eta=-H_N\,\Rightarrow\,-\lambda_o a_o-a^i p_i=-\lambda_o X^o - p_i X^i
$$
and the second one
$$
\mathbf i(X_N)\d\,\eta=\d\, H_N-(L(R)H_N)\eta\,\Rightarrow\,$$ $$-b_i\d\,x^i+ji\d\,p_i=
\lambda_o\frac{\partial X^o}{\partial x^i}\d\,x^i+X^i\d\,p_i+
p_j\frac{\partial X^j}{\partial x^i}\d\,x^i \, .
$$
Hence 
$$
a^i=X^i,\,\, b_i=-\lambda_o\frac{\partial X^o}{\partial x^i}-
p_j\frac{\partial X^j}{\partial x^i},\,\, a^o=X^o
$$
as we wanted.

\bigskip
\noindent\textbf{b) The abnormal situation: $p_o=0$}

\medskip
This case corresponds to $\lambda_o=0$ and the submanifold $N_o\subset T^*M$ defined by $p_o=0$. Let $j_o:N_o\hookrightarrow T^*M$ be the natural inclusion and $\eta_o=j_o^*\theta_M$.

Observe that $\eta_o=-p_i\d\,x^i$ is not a contact form. In fact, as $m_o=\dim M_o$, we have that $\eta_o\wedge(\d\,\eta_o)^{m_o-1}\neq 0$, but $\eta_o\wedge(\d\,\eta_o)^{m_o}=0$.

We can consider the 2-form $\omega_o=\d\, \eta_o$, the Hamiltonian $H_o=j_o^*H$ and the presymplectic manifold $(N_o,\omega_o,H_o)$. Observe that $\ker\omega_o=\{\frac{\partial}{\partial x^o}\}$. The Hamiltonian presymplectic equation  
$$
\mathbf{i}(X_o)\omega_o= \d\,H_o
$$  
gives the solution
$$
X_o= X^i\frac{\partial}{\partial x^i}-
p_j\frac{\partial X^j}{\partial x^i}\frac{\partial}{\partial p_i}+A\,\frac{\partial}{\partial x^o}\,,
$$ 
where $A$ is arbitrary and corresponds to $\ker\omega_o$. In fact we have that $\dot{x}^o=A$.

It does not exist any constraint because the vector field $X_o$ is defined  on the whole manifold $N_o$. This is because the only constraint is given by $L_T H_ o=0$ with $T\in \ker\omega_o$ and this is fulfilled globally on $N_o$.

\bigskip
\noindent\textbf{Comment}: Observe that $T^*M=\bigcup_{\lambda\in\Real}N_{\lambda}\,$, hence with these decomposition we obtain all the solutions of the initial Hamiltonian problem on $T^*M$ given by the Hamiltonian $H$. 

\section{The contact dynamics approach to Pontryagin Maximum Principle}\label{codynpmp}

\bigskip
Following the ideas of the previous sections, we study a contact approach to the Pontryagin Maximum Principle, in particular to the so called normal solutions to the optimal control problem. In particular we will obtain the normal solutions of an optimal control problem as projection of the integral curves of a Hamiltonian contact system in adequate manifolds. The abnormal solution can be obtained with another different approach given at the end of this section.

\subsection{Statement of the problem}
Let $(M,U,X,I,x_a,x_b)$ be an optimal control problem. We know by Theorem \ref{presym-pmp} that to solve this problem we need to study the associated Hamiltonian presymplectic system $(T^*M\times U,\omega,H)$, that is to obtain an integral curve of the vector field $X_H$ solution to the equation $\mathbf{i}(X_H)\,\omega=\d\, H$, where
$$
\omega=\pi_1^*\omega_o=\d x^o\wedge\ p_o+\d x^i\wedge\d p_i,\quad H=\hat X=p_oF+p_iX^i
$$
and $\pi_1:TM^*\times U\to TM^*$. Recall that $\ker\omega=\left\{\partial/\partial u^a\right\}$.

The solution to the equation $\mathbf{i}(X_H)\,\omega=\d\, H$ is given by:
\begin{equation} X_H=F\frac{\partial}{\partial x^o}+X^i\frac{\partial}{\partial x^i}-
\left(\lambda_o\frac{\partial F}{\partial x^i}+
p_j\frac{\partial X^j}{\partial x^i}\right)\frac{\partial}{\partial p_i}+A^a\frac{\partial}{\partial u^a}\, .
\end{equation}

Observe that this solution exists all over the manifold
$T^*M\times U$ and that $p_o$ is constant for every curve solution to the problem. The last term corresponds to the elements of $\ker\omega$.

The minimality, compatibility, conditions are $\frac{\partial H}{\partial u^a}=0$ for every $a$, are used to determine the controls.

As we said in Section 3, if the compatibility equations allows us to determine the controls $u^1,\ldots,u^k$, that is we can obtain $u^a=	\psi(x^o,x^i,p_o,p_i)$, then we say that the optimal control problem is \textbf{regular}, otherwise it is called \textbf{singular}. In the singular case, it is necessary to apply an algorithm of constraints, that is to go to higher order conditions, to obtain the controls perhaps on a submanifold o $T^*M\times U$. Suppose that we are in the regular situation, hence we have determined the controls by the compatibility conditions.

With the regularity assumption as the controls $u^a$ has been determined, we have that $X_H$ is projected to the manifold $T^*M$ and has components only in $(x^o,x^i,p_o,p_i)$.Then we have:
\begin{equation} X_H=F\frac{\partial}{\partial x^o}+X^i\frac{\partial}{\partial x^i}-
\left(\lambda_o\frac{\partial F}{\partial x^i}+
p_j\frac{\partial X^j}{\partial x^i}\right)\frac{\partial}{\partial p_i}\end{equation}
because we are in the symplectic case.

We know that, for all the solutions of the associated presymplectic formulation, we have that the moment  $p_o(t)$ is a constant. Following the previous section,  we will try to classify the solutions according to the real value of $p_o$. Hence we define and study

\begin{enumerate}
  \item[a)] \textbf{Normal solutions}: those with $p_o=\lambda_o\neq 0$.
  \item[b)] \textbf{Abnormal solutions}: those with $p_o=\lambda_o=0$.  
\end{enumerate}

\subsection{Normal solutions:  \texorpdfstring{$p_o=\lambda_o\neq 0$}{}}

Let $N\subset T^*M$ be the submanifold defined by $p_o=\lambda_o$ and $j:N\hookrightarrow T^*M$ be the natural inclusion.  We denote by $(x^o,x^i,p_i)$ the coordinates induced in $N$ by the coordinates we have in $T^*M$.

Consider now the canonical 1-form $\theta_M\in\Omega^1(T^*M)$ and let $
\eta=-(j)^*\theta_M$, then we have that

\begin{lem}$(N,\eta)$ is a contact manifold. The Reeb vector field is $R=-\frac{1}{\lambda_o}\frac{\partial}{\partial x^o}$.	
\end{lem}
 
 Let $H_N=(j)^*H$ the restriction of $H$ to $N$, then $H_N= \lambda_o(X)^o+p_i(X)^i$ and we have a Hamiltonian contact system given by $(N,\eta,H_N)$. Let $X_N\in\mathfrak{X}(N)$ be the corresponding contact Hamiltonian vector field, that is the solution to the equations
 $$
 \mathbf i(X_N)\eta=-H_N,\,\,\,\mathbf i(X_N)\d\,\eta=\d\, H_n-(L(R)H_N)\eta\, 
 $$
 whose local expression is
\begin{equation}\label{norm-cont-u}
 X_N=X^o\frac{\partial}{\partial x^o}+X^i\frac{\partial}{\partial x^i}-
\left(\lambda_o\frac{\partial X^o}{\partial x^i}+
p_j\frac{\partial X^j}{\partial x^i}\right)\frac{\partial}{\partial p_i}\, .
\end{equation}
With the usual notation denoting by the same names the functions on $T^*M$ and their restrictions to $N$.

With this in mind and following Section \ref{sympl-ext}, we have that:

\begin{prop}\

 The vector field $X_H\in\mathfrak{X}(T^*M)$ is tangent to $N$ and, on the points of $N$, it is equal to $X_N$.	
 \end{prop}
 
 Hence, for every $u\in U$, all the normal solutions to the optimal control problem are solutions to a contact Hamiltonian problem.
 
\subsection{Abnormal solutions:  \texorpdfstring{$p_o=\lambda_o=0$}{}}

Let $N_o\subset T^*M$ the submanifold defined by $p_o=0$. Let $j_o:N_o\hookrightarrow T^*M$ be the natural inclusion and $\eta_o =j_o^*\,\theta_M$.

As above, $\eta_o=-p_i\d\,x^i$ is not a contact form and we have that $\eta_o\wedge(\d\,\eta_o)^{m_o}=0$.

We can consider the 2-form $\omega_o=\d\,\eta_o$, 
the Hamiltonian $H_o=j_o^*H$ and the 
presymplectic manifold $(N_o,\omega_o,H_o)$. 
Observe that $\ker\omega_o=\{\frac{\partial}{\partial x^o}\}$. The Hamiltonian presymplectic equation  
$$
\mathbf{i}(X_o)\omega_o= \d\,H_o
$$  
gives the solution
$$
X_o= X^i\frac{\partial}{\partial x^i}-
p_j\frac{\partial X^j}{\partial x^i}\frac{\partial}{\partial p_i}+A^a\,\frac{\partial}{\partial u^a}\,,
$$
where $A^a$ are arbitrary and correspond to $\ker\omega_o$. 

And it does not exist any constraint because the vector field $X_o$ is defined  on the whole manifold $N_o$.

\textbf{Note}: We can also solve the precontact problem given by $(N_o^u,\eta_o^u,H_o^u)$.

\bigskip
\noindent\textbf{Comment}: Observe that $T^*M=\bigcup_{\lambda\in\Real}N_{\lambda}\,$, hence with this decomposition we obtain all the solutions of the  Hamiltonian problem on $T^*M$ given by the Hamiltonian $H$. Some of them, the normal solutions, as contact problems, and the abnormal solutions as symplectic ones.

With all this in mind, we have proved the following

\begin{teor}: \textbf{Contact Pontryagin Maximum Principle}\label{contact-pmp}

Consider the optimal control problem $(M,U,X,I,x_a,x_b)$, with $U\subset\Real^k$ an open set. 
Let $\hat\sigma:I\to T^*M\times U=T^*\Real\times T^*M_o\times U$, $\hat\sigma=(\sigma_{T^*M},\sigma_U)$, be a solution of the presymplectic Pontryagin Maximum Principle for such problem and suppose we are in the \textbf{regular} case, that is the minimality  conditions $(\partial H/\partial u)=0$, for every $u\in U$ and for every $t\in I$ allows to determine the controls.
 Then
 \begin{enumerate}
  \item[a)] if $\hat\sigma$ is a \textbf{normal solution} with $p_o=\lambda_o\neq 0$, then s it is an integral curve of the \textbf{contact Hamiltonian system} $(N,\eta,H_N)$, as described above, with $H_N=\lambda_o F+p_iX^i$. 

\item[b)] if $\hat\sigma$ is an \textbf{abnormal solution}, then it is an integral curve of the \textbf{presymplectic Hamiltonian system} $(N_o,\omega_o,H_o)$, as described above, with $H_o=p_i\, X^i$ . 
\end{enumerate}
\end{teor}

\bigskip
For the \textbf{normal solutions}, they satisfy the differential equations:
$$
\dot x^o=\frac{\partial H}{\partial p_o}=F,\,\,\dot p_o=\frac{\partial H}{\partial x^o}=0,\,\,(\Rightarrow p_o=ct) 
$$
$$
\dot x^i=\frac{\partial H}{\partial p_i}=X^i,\,\, 
\dot p^i=-\frac{\partial H}{\partial x^i}=-p_o\frac{\partial F}{\partial x^i}-p_j\frac{\partial X^j}{\partial x^i}
$$
$$
\frac{\partial H}{\partial u^1}=0,\ldots,\frac{\partial H}{\partial u^k}=0 
$$
where $H=\lambda_oF+p_iX^i$ with $\lambda_o\neq 0$.

For the \textbf{abnormal solutions}, the corresponding differential equations are
$$
\dot x^i=\frac{\partial H}{\partial p_i}=X^i,\,\, 
\dot p^i=-\frac{\partial H}{\partial x^i}=-p_j\frac{\partial X^j}{\partial x^i}
$$
$$
\frac{\partial H}{\partial u^1}=0,\ldots,\frac{\partial H}{\partial u^k}=0 
$$
where $H=p_iX^i$.

\section{Herglotz variational problem as an optimal control problem}\label{Hvp-ocp}
In Section \ref{Hvp} we have studied the Herglotz variational principle; there we obtained the contact equations for a Hamiltonian contact system as solution of a variational problem but with a generalization of the Hamilton variational principle. This more general principle was stated and solved in 1930 by Gustav Herglotz, see \cite{He-1930,GuGoGu-96}. The idea was to change the integral statement on the curves solution to the problem by a differential equation defined precisely by the Lagrangian function. Interest in this approach has been increasing since the last referred publication and its relation with contact dynamics and dissipation systems, see for example \cite{GeGu2002, Bravetti2017a} and references therein. In this Section we approach Herglotz principle as an optimal control problem and find the corresponding differential equations, the generalized Euler--Lagrange equations, with a new proof through the Pontryagin Maximum Principle.

\subsection{Statement of the problem}\label{herglotz_problem}
We begin recalling the statement of the Herglotz variational problem as we did in Section \ref{Hvp}.

Let $Q$ be a smooth manifold and  $F:TQ\times\Real\to\Real$ a smooth function and consider the following problem:

\bigskip
\noindent\textbf{Herglotz variational problem}: Find curves $\Gamma=(\gamma,\zeta):I=[a,b]\to Q\times\Real$, such that
\begin{enumerate}
  \item[1)] end points conditions: $\gamma(a)=q_a, \gamma(b)=q_b,\zeta(a)=0$, 
  \item[2)] $\Gamma$ is an integral curve of $\dot{z}=F(q,v,z)$: $\dot{\zeta}=F(\gamma(t),\dot\gamma(t),\zeta(t))$, for every $t\in I$, and
  \item[3)] extreme condition: $\zeta(b)$  is maximum over all curves satisfying 1) and 2). 
  \end{enumerate}
  
 Observe that we have considered the differential equation $\dot{z}=F(q,v,z)$ depending on the curves $\gamma$. In the case that the function $F$ does not depend on the variable $z$, that is $F:TQ\to\Real$, then the diferential equations is $\dot z=F(\gamma,\dot\gamma)$, hence by integration, the problem is the classical variational one defined by: find the curves $\gamma(t)$ minimizing 
 $$
 S[\gamma]=\int_a^b F(\gamma(t),\dot\gamma(t))\,\d\, t
 $$
with initial conditions $\gamma(a)=q_a, \gamma(b)=q_b$.

As we know, Herglotz obtained that the curves $\gamma$ solution to this problem satisfy the so called generalized Euler--Lagrange equations
$$
 \frac{\d}{\d t}\left(\frac{\partial F}{\partial v^i}\right)_\gamma-\frac{\partial F}{\partial q^i}
-\frac{\partial F}{\partial z}\frac{\partial F}{\partial v^i}=0\, .
 $$
In this section we will obtain these differential equations as an application of the Pontryagin Maximum Principle to a suitable optimal control problem associated to the Herglotz variational problem.

To do so, we begin by giving a geometric statement of the Herglotz problem. Given the function  $F:TQ\times\Real:\to\Real$, consider the right up triangle of the following diagram 
$$
\xymatrix{&  &T\Real\ar[d]^{\tau_o}\\
&{TQ\times \Real}\ar[d]^{\tau_Q\times I_{\Real}}\ar[r]^{\,\,\,\,\pi_2}\ar[ru]^Z&\Real\\
I\ar[r]^{\Gamma}\ar[ru]^{\hat\Gamma}& Q \times\Real & }
$$
where $Z\in\vf(\Real,\pi_2)$ is the vector field on $\Real$ along the projection $\pi_2$ defined by
$$
Z=F \,\frac{\partial }{\partial z}\, .
$$
Now taking the full diagram, we have the following problem associated with the vector field $Z$ 

\bigskip
\noindent\textbf{Geometric Herglotz variational problem}: Find curves $\Gamma:I=[a,b]\to Q\times\Real$,  $\Gamma=(\gamma,\zeta)$, such that
\begin{enumerate}
  \item[1)] end points conditions: $\Gamma(a)=(q_a,0) , \gamma(b)=q_b$, 
  \item[2)] $\Gamma$ is an integral curve of $Z$: $\dot{\zeta}=F(\tilde\Gamma(t))$, for every $t\in I$, where $\tilde\Gamma=(\gamma'=(\gamma,\dot\gamma),\zeta)$, and
  \item[3)] extreme condition: $\zeta(b)$  is maximum over all curves satisfying 1) and 2). 
  \end{enumerate}
Obviously the two above  problems are equivalent. The difference is only in the language used to state them.

\subsection{Optimal control approach to the Herglotz variational problem}

Associated to the function $F:TQ\times\Real:\to\Real$,  consider the following diagram
$$
\xymatrix{&T(Q\times\Real)\ar[d]^{\tau_{Q\times\Real}}\\
{TQ\times \Real}\ar[r]^{\tau_Q\times I_\Real}\ar[ru]^Y&Q\times\Real }
$$
where $Y$ is the vector field on $Q\times \Real$ along the projection $\tau_Q\times I_\Real:TQ\times \Real\to Q\times \Real$ defined by  $Y((q,v),z)=((q,z),v,F))$, which in local coordinates  $Y$ is given by
$$
Y=v^i\frac{\partial}{\partial q^i}+F\frac{\partial}{\partial z}\, .
$$
This vector field corresponds to the system of ordinary differential equations:
$$
\dot q^i=v^i,\qquad \dot z=F(q^i, v^i,z)\, .
$$
Observe that the first sumand of the vector field $Y$ is a canonical vector field along the projection $\tau_Q:TQ\to Q$, it corresponds to the identity map  $I_{TQ}:TQ\to TQ$. Hence the vector field $Y$ is associated in a natural way to the function $F$.

These elements define a \textit{control system} with \textit{vector field }$Y\in\vf(Q\times\Real, \tau_Q\times I_\Real)$, on the \textit{state space} $Q\times\Real$, and with the fibres of $TQ$ as the \textit{set of controls}; that is for every state $(q,z)\in Q\times\Real$, the controls are the elements $v\in T_q Q$.

On this control system we state the following optimal control problem: Consider the diagram
$$
\xymatrix{
&T(Q\times\Real)\ar[d]_{\tau_{Q\times\Real}}\\
TQ\times \Real\ar[r]^{\tau_Q\times I_\Real}\ar[ru]^Y&Q\times\Real
\\
& I\ar[u]^\Gamma\ar[ul]^{\tilde\Gamma}\ar@/_7mm /[uu]_{\Gamma'}}
$$
where, if $\Gamma=(\gamma,\zeta)$, then $\tilde\Gamma=(\gamma',\zeta)=((\gamma,\dot\gamma),\zeta)$.

For a curve $\Gamma:I\to Q\times\Real$, we take its canonical lifting  to the tangent bundle, $\Gamma':I\to T(Q\times\Real)$, that is:  if $\Gamma=(\gamma,\zeta)$ then $\Gamma'= ((\gamma,\zeta), (\dot\gamma, \dot\zeta))$. 

We say that a curve $\Gamma$ is an \textbf{integral curve} of the vector field $Y$ if:
$$
\Gamma'=Y\circ\tilde\Gamma\, ,\qquad  (\dot\gamma,\dot\zeta)=Y(\gamma,\dot\gamma,\zeta=(v^i(\gamma,\dot\gamma),F(\gamma,\dot\gamma,\zeta)))\, ,
$$
which, in local coordinates, is a solution to the above system of differential equations:
$$
\dot q^i=v^i,\qquad \dot z=F(q^i, v^i,z)\, .
$$

Hence we have the optimal control problem given by:

\bigskip
\noindent\textbf{Optimal control problem associated to Herglotz variational problem}: 

\medskip
\noindent Find curves $\Gamma:I=[a,b]\to Q\times\Real$,  $\Gamma=(\gamma,\zeta)$, such that
\begin{enumerate}
  \item[1)] end points conditions: $\Gamma(a)=(q_a,0) , \gamma(b)=q_b$, 
  \item[2)] $\Gamma$ is an integral curve of $Y$: $\Gamma'(t)=Y(\tilde\Gamma(t))$, for every $t\in I$ and
  \item[3)] optimal condition: $\zeta(b)$  is maximum over all curves satisfying 1) and 2). 
  \end{enumerate}
 Observe that the optimal condition can be stated as:
 $$
 {\mathrm max}\,\, z(b)={\mathrm max} \int_a^b\dot z(t)\d t={\mathrm max} \int_a^bF(q(t),\dot q(t),z(t))\d t\, ,
 $$
 hence we have a classical optimal control theory with $F$ as the cost function. 
 
  This optimal control problem, which can be solved using the Pontryagin Maximum Principle, is \textbf{equivalent} to the above Herglotz variational problem: if $\Gamma=(\gamma,\zeta)$ is a solution to the above optimal control problem then $\gamma$ is a solution to the Herglotz variational problem and $\dot\zeta=F(\gamma,\dot\gamma,\zeta)$, and conversely.
  
  We denote this problem by $(M,U,X,I,x_a,x_b)=(Q\times\Real, TQ, Y,I,q_a,q_b)$ with the notation described in Section 3.
  
  \subsection{Application of the presymplectic form of the Pontryagin Maximum Principle}
  
  According to Section 3, first we have to extend the problem and declare the direction  where the optimization must be done using the cost function.
  
  \subsubsection{The extended problem}
 Observe that  in the above optimal control problem, $(Q\times\Real, TQ, Y,I,q_a,q_b)$, the cost function is $F$, that corresponds also to the state variable $z$, then we need to extend the problem adding a new variable with $F$ as derivative. Denote by $q^o\in\Real$ this new variable. The differential equation to add to the system is $\dot q^o= F(q,v,z)$. To change to this extended problem we need to consider the diagram
 
  $$
\xymatrix{&T(\Real\times Q\times\Real)\ar[d]^{\tau}\\
{\Real\times TQ\times \Real}\ar[r]^{I_{\Real}\times\tau_Q\times I_\Real}\ar[ru]^{\hat Y}&\Real\times Q\times\Real }\, ,
$$

\medskip
\noindent and take the control system given by the \textbf{dynamical vector field} $\hat Y\in\vf(\Real\times Q\times\Real,I_{\Real}\times \tau_Q\times I_\Real)$, which in coordinates reads
  $$
  \hat Y=F\frac{\partial}{\partial q^o}+v^i\frac{\partial}{\partial q^i}+F\frac{\partial}{\partial z}\, ,
  $$
  with the manifold $\Real\times Q\times\Real$ as \textbf{state space}  and with \textbf{controls} the fibres of $TQ$, that is for every state $(q^o,q,z)\in \Real\times Q\times\Real$, the controls are the elements of $v\in T_q Q$.

On this system,  the precise statement of the optimal control problem we have is:
 
 \bigskip
\noindent \textbf{Extended optimal control formulation of the Herglotz variational problem}:  

\medskip
\noindent Find curves $\hat\Gamma:I=[a,b]\to\Real\times Q\times\Real$,  $\Gamma=(\gamma^o,\gamma,\zeta)$, such that
\begin{enumerate}
  \item[1)] end points conditions: $\hat\Gamma(a)=(0,q_a,0) , \gamma(b)=q_b$, 
  \item[2)] $\hat\Gamma$ is an integral curve of $\hat Y$: $\hat\Gamma'(t)=\hat Y(\tilde{\hat\Gamma}(t))$, for every $t\in I$, where  $\tilde{\hat\Gamma}:I\to\Real\times TQ\times\Real$, $\tilde{\hat\Gamma}=(\gamma^o,((\gamma, \dot\gamma),\zeta)$,and
  \item[3)] extreme condition: $\zeta(b)$  is maximum over all curves satisfying 1) and 2). 
  \end{enumerate}
  
  This is the optimal control problem denoted by $(\Real\times Q\times\Real, TQ, \hat Y,I,q_a,q_b)$.
  
  \subsubsection{Solution of the extended problem with the presymplectic form of the Pontryagin Maximum Principle}

Following Section 3, to solve this optimal control problem consider the projection 
$$\hat\pi_1: T^*(\Real\times  Q\times\Real)\times TQ\to T^*(\Real\times  Q\times\Real)\,\, .$$

This last manifold has a canonical symplectic form $\omega_{\Real \times Q \times \Real} \in\Omega^2(T^*(\Real\times  Q\times\Real))$, $\omega_{\Real \times Q \times \Real}=-\d\theta_{\Real \times Q \times \Real}$, which in canonical coordinates, $(q^o,p_o,q^i,p_i,z,p_z)$, reads
$$
\omega_{\Real \times Q \times \Real}=-\d \theta_{\Real \times Q \times \Real}=-\d(p_o\d q^o+p_i\d x^i+p_z\d z)=\d q^o\wedge\d p_o+\d x^i\wedge\d p_i+\d z\wedge\d p_z\,\, .
$$
Let $\omega=\hat\pi_1^*\,\omega_{\Real \times Q \times \Real}$, then $\omega
$ is a presymplectic form in 
$T^*(\Real\times  Q\times\Real)\times TQ$, its kernel being the vector fields  tangent to $TQ$ which are vertical vector fields, that is tangent to the fibres of $\tau_Q:TQ\to Q$. Hence $\ker \omega$ is locally generated by
$$
\frac{\partial}{\partial v^1},\ldots,\frac{\partial}{\partial v^n}\,\, ,
$$
if $\dim Q=n$. The local expressions of $\omega$ and $\omega_o$ are the same, with the usual abuse of notation for the local coordinates.

With the vector field $\hat Y$, as usually,  we can built a natural Hamiltonian function given by $H:T^*(\Real\times  Q\times\Real)\times TQ\to\Real$, locally given as
$$
H(q^o,p_o,x^i,p_i,z,p_z,v^i)=p_oF+p_iv^i+p_z F\, ,
$$
and consider the presymplectic system $(T^*(\Real\times  Q\times\Real)\times TQ,\omega,H)$.

The corresponding Hamiltonian vector field $X_H$, satisfying the equation $\textbf{i}_{X_H}\omega=\d H$, is locally given by
\begin{eqnarray}
X_H&=&F\frac{\partial}{\partial q^o}+0\frac{\partial}{\partial p_o}+
v^i\frac{\partial}{\partial q^i}+F\frac{\partial}{\partial z}
\\
&-&\left(p_o\frac{\partial F}{\partial q^i}+p_z\frac{\partial F}{\partial q^i}\right)
\frac{\partial }{\partial p_i}-
\left(p_o\frac{\partial F}{\partial z}
+p_z\frac{\partial F}{\partial z}\right)
\frac{\partial }{\partial p_z}+A^i\frac{\partial}{\partial v^i}\,\, ,
\end{eqnarray}

\noindent where the last term corresponds to the kernel of $\omega$.

The compatibility conditions for the presymplectic system, or the optimality conditions, are given  by, see \cite{got79,MURo-92},
$$
L_VH=0
$$
for every $V\in \ker\omega$, when restricted to the curves $\sigma=(\sigma^o, \delta_0,\sigma^i,\delta_i, \sigma^z,\delta_z, w^i)$, solution to the system of differential equations
\begin{eqnarray}
\dot q^o=F,&\dot p_o=0\\
\dot q^i=v^i,&
\dot p_i=-(p_o+p_z)\frac{\partial F}{\partial q^i}\\
\dot z=F, &\dot p_z=-(p_o+p_z)\frac{\partial F}{\partial z}\\
\dot v^i=A^i
\end{eqnarray}
where the $A^i$ are free. These differential equations correspond to the integral curves of the vector field $X_H$.

In local coordinates, the compatibility conditions are $L_{\frac{\partial}{\partial v^i}}H=0$, for every $i=1,\ldots, n$. As $H=(p_o+p_z)F+p_iv^i$, we have:
$$
L_{\frac{\partial}{\partial v^i}}H=(p_o+p_z)\frac{\partial F}{\partial v^i}+p_i=0\, .
$$
In the weak presymplectic Pontryagin Maximum Principle, these are the conditions from where we can obtain the controls $v^i$, looking for the critical points of $H$ with respect to the controls.

 In the present situation, these functions are constraints defining a submanifold of $T^*(\Real\times  Q\times\Real)\times TQ$, and the Hamiltonian vector field solution, $X_H$, have to be tangent to this submanifold, hence:
 $$
 L_{X_H}\left((p_o+p_z)\frac{\partial F}{\partial v^i}+p_i\right)=0\, ,
 $$

 \noindent but
$$
 L_{X_H}\left((p_o+p_z)\frac{\partial F}{\partial v^i}+p_i\right)=
 \!\!(p_o+p_z)\!\left(
v^j\frac{\partial^2 F}{\partial q^j\partial v^i}+F\frac{\partial^2 F}{\partial z\partial v^i}
-\frac{\partial F}{\partial q^j}-\frac{\partial F}{\partial z}\frac{\partial F}{\partial v^j}+A^j\frac{\partial^2 F}{\partial v^j\partial v^i}\right)\, ,
$$
 where $A^j= \ddot v^j$. Hence we have:

 $$
  (p_o+p_z)\left(
v^j\frac{\partial^2 F}{\partial q^j\partial v^i}+F\frac{\partial^2 F}{\partial z\partial v^i}
-\frac{\partial F}{\partial q^j}
- \frac{\partial F}{\partial z}\frac{\partial F}{\partial v^j}+\ddot v^j\frac{\partial^2 F}{\partial v^j\partial v^i}\right)=0\, .
 $$

 Which on the curves solution gives
 \begin{eqnarray*}
 L_{X_H}\left((p_o+p_z)\frac{\partial F}{\partial v^i}+p_i\right)
=(p_o+p_z)\left(\frac{\d}{\d t}\frac{\partial F}{\partial v^i}-\frac{\partial F}{\partial q^i}
-\frac{\partial F}{\partial z}\frac{\partial F}{\partial v^i}\right)=0\, .
\end{eqnarray*}

 These differential equations are a necessary condition, for a curve $\sigma$ on the manifold $T^*(\Real\times Q\times\Real)\times TQ$, to be solution of  the presymplectic system have to satisfy when it is projected to $Q\times\Real$ .
 
 But we have that
 
 \begin{lem}
 On the solution curves the quantity $p_o+p_z$ is not null unless may be on some isolated points.
  \end{lem}
\begin{proof}: We know that $\dot p_z=-(p_o+p_z)\frac{\partial F}{\partial z}$
  and $\dot p_o=0$, hence $p_o$ is a constant, then the differential equation defining $p_z$ is
  $$
  \dot{p}_z=-p_z\frac{\partial F}{\partial z}-p_o\frac{\partial F}{\partial z}=-Ap_z-p_oA\, ,
  $$
  where, on the solution curves, $A$ is a function of $t$. This is a linear differential equation whose general solution is
  $$
  p_z(t)=p_o+\lambda \exp{(-\int A)}\,,\quad \lambda\in\Real\,   , \,\, p_o\in\Real,
  $$
 and the proof is finished.
 \end{proof}
 
 \subsection{The final results}
 
 From the above Lemma we obtain that:
 
 \begin{teor}
 If 	$\sigma=(\sigma^o, \delta_0,\sigma^i,\delta_i, \sigma^z,\delta_z, w^i)$ is a solution to the presymplectic system  $(T^*(\Real\times  Q\times\Real)\times TQ,\omega,H)$, then its projection to $Q\times\Real$,  $(\sigma^i,\sigma^z)$, satisfyes  the equations
 $$
 \frac{\d}{\d t}\frac{\partial F}{\partial v^i}-\frac{\partial F}{\partial q^i}
-\frac{\partial F}{\partial z}\frac{\partial F}{\partial v^i}=0
 $$
 \end{teor}

 Hence if we include in the statement the original problem, we have proven the following
 \begin{teor}
 If $\sigma=(\sigma^o, \delta_0,\sigma^i,\delta_i, \sigma^z,\delta_z, w^i)$, $\sigma:I\to T^*(\Real\times  Q\times\Real)\times TQ$,  is a solution to the presymplectic system $(T^*(\Real\times  Q\times\Real)\times TQ,\omega,H)$, then  
 
 \begin{enumerate}
  \item[a)] its projection to $\Real\times Q\times\Real$, $\hat\Gamma:I=[a,b]\to\Real\times Q\times\Real$,  $\Gamma=(\gamma^o=\sigma^o,\gamma=(\sigma^i),\zeta=\delta_z)$, is a solution to the extended optimal control problem	$(\Real\times Q\times\Real, TQ, \hat Y,I,q_a,q_b)$
\item[b)] its projection to $Q\times\Real$, $\Gamma:I=[a,b]\to Q\times\Real$,  $\Gamma=(\gamma=(\sigma^i),\zeta=\delta_z)$, is a solution to the optimal control problem $(Q\times\Real, TQ, Y,I,q_a,q_b)$.
\end{enumerate}
 \end{teor}
 
 \bigskip
As we know, this optimal control problem is equivalent to the Herglotz variational problem given by  the function $F:TQ\times\Real:\to\Real$, the interval $I=[a,b]$ and the initial conditions $q_a,q_b$, then we have proven the

\begin{teor}
Given the manifold $Q$ and the function $F:TQ\times\Real:\to\Real$. If
$\Gamma=(\gamma,\zeta)$ is a solution to the Herglotz variational problem defined by  $F$, then $\Gamma$ satisfies the differential equations
$$
 \frac{\d}{\d t}\frac{\partial F}{\partial v^i}-\frac{\partial F}{\partial q^i}
-\frac{\partial F}{\partial z}\frac{\partial F}{\partial v^i}=0\, ,
 $$
  which are known as \textbf{generalized Euler-Lagrange equations} for the Herglotz problem.
\end{teor}

See \cite{GuGoGu-96,GeGu2002} for comparison between different proofs.
\section{Herglotz optimal control problem}
In this section we give a generalization of the classical optimal control problem in the same way that Herglotz variational problem is a generalization of Hamilton principle in mechanics.
\subsection{Statement of the problem}
As it was described in Section 3, a classical optimal control problem is given by the elements $(M,U,X,F,I,x_a,x_b)$. The cost function $F:M\times U\to\Real$ is used to express the optimization condition as an integral 
$$
S[\gamma(t)=(x(t),u(t))]=\int_a^b F(x(t), u(t))\d\,t\, ,
$$
which is a functional on the curves $\gamma:I\to M\times U$, satifying some initial conditions and being integral curves of the vector field $X$, that is $\dot{x}=X(x(t),u(t))$.

This is ``similar" to the classical variational calculus with $F$ in the role of the Lagrangian and the integrability condition for the curves as a contraint. 

But the generalization proposed and studied by Herglotz changes the integral functional as  the element to optimize by a diferential equation  satisfied by a new variable, denoted by $z$, differential equation just defined by the cost function $F$, that is $\dot z=F(x,u,z)$, instead of the above integral; see Sections \ref{Hvp} or \ref{Hvp-ocp} for more details.

Now we propose a generalization of the classical optimal control problem following the ideas of Herglotz.

Remembering the elements giving us an optimal control problem, we have the diagram
$$
\xymatrix{
  & TM\ar[d]^{\tau_M}\\
M\times U\ar[ru]^X \ar[r]^\pi & M\, ,  }$$
that is $X\in\vf(M,\pi)$, $X(x,u)$, and a cost  function, $F(x,u)$, to integrate on the curves solution to the differential equation given by $X$. Instead of this cost function, we take a function $F:M\times\Real\times U\to\Real$, depending also on a new variable $z$, and consider the following problem:

\bigskip
\noindent\textbf{Herglotz optimal control problem}:

Find curves $\gamma:I=[a,b]\to M\times \Real\times U$, $\gamma=(\gamma_M,\gamma_z,\gamma_U)$, such that
\begin{enumerate}
  \item[1)] end points conditions: $\gamma_M(a)=x_a, \gamma_M(b)=x_b, \gamma_z(a)=0$,
  \item[2)] $\gamma_M$ is an integral curve of $X$: $\dot{\gamma}_M=X\circ(\gamma_M,\gamma_U)$\, , 
  \item[3)] $\gamma_z$ satifies the differential equation $\dot z=F(x,z,u)$, and  
  \item[4)] maximal condition: $\gamma_z(b)$ is maximum over all curves satisfying 1), 2) and 3). 
  \end{enumerate}
  
  The differential equations corresponding to this problem are
  \begin{equation}\label{Hocp-difeq}
  \dot x^i=X^i(x,u),\qquad \dot z=F(x,z,u)\, .
  \end{equation}
  
  If the function $F$ does not depend on $z$, then the maximal condition takes the form
  $$
  \hspace{-30mm} 4')\,\, z(b)= \int_a^b F(x,u)\d t\quad  \mathrm{ is\,\,  maximum}
  $$
 which gives a classical optimal control problem. Hence we have a generalization of the classical problem in the sense of Herglotz.
 
 In order to solve this problem we begin by transforming it into a classical optimal control problem.

\subsection{Solution to Herglotz optimal control problem}\label{sec:hocp_vf}

There is another way to organize all these elements, $(M,U,X,F,x_a,x_b)$, in a shorter form. Consider the following diagram
$$
\xymatrix{
&T(M\times\Real)\ar[d]^{\tau_{M\times\Real}}\\
{M\times\Real\times U}\ar[ur]^{Z}\ar[r]^{\quad\pi}&M\times\Real\\
I\ar[u]^\gamma\ar[ru]^{\Gamma}
}
$$
where  $Z=X+Y$, that is  $Z=X^i\frac{\partial}{\partial x^i}+F\frac{\partial}{\partial z}$, locally. And the curves are $\Gamma=(\gamma_M,\gamma_z)$, $\gamma= (\Gamma,\gamma_U)= (\gamma_M,\gamma_z,\gamma_U)$.

\medskip
Then we have another equivalent statement:

\bigskip
\noindent\textbf{Herglotz optimal control problem}: Find curves $\gamma:I=[a,b]\to M_o\times \Real\times U$, $\gamma=(\gamma_M,\gamma_z,\gamma_U)$,  $\Gamma=(\gamma_m,\gamma_z)$ such that
\begin{enumerate}
  \item[1)] end points conditions: $\gamma_M(a)=x_a, \gamma_M(b)=x_b, \gamma_z(a)=0$,
  \item[2)] $\Gamma_M$ is an integral curve of $Z$: $\dot{\Gamma}_M=Z\circ\gamma$\, ,  and
  \item[3)]  optimal condition: $\gamma_z(b)$ is maximum over all curves satisfying 1),
  and 2). 
  \end{enumerate}\label{hocp}

Condition 2) is written as $(\dot\gamma_,\dot\gamma_z)=Z\circ\gamma$, that is
\begin{equation}
  \dot x^i=X^i(x,u),\qquad \dot z=F(x,z,u)\, .
  \end{equation}
which are the same set of differential equations as equations (\ref{Hocp-difeq}). Hence both problems are equivalent. In the sequel we refer to this second form.

Observe that with this  approach, we have a classical optimal control problem and we can find its solution following the method of the Section 3, in particular by applying the weak presymplectic form of the Pontryagin Maximum Principle, Theorem \ref{presym-pmp}. In this case, the function to optimize is one of the directions of sate space which is given by $z$.

We begin, as usual, by extending the vector field, hence obtaining the extended system adding a new variable $x^o$ for the variable $z$ to maximize. The new vector field is
$$
\overline{X}=  F\frac{\partial}{\partial x^o}+X^i\frac{\partial}{\partial x^i}+F\frac{\partial}{\partial z}\in\vf(\Real\times M\times\Real)\, .
$$
Then the associated Hamiltonian is $H(x^o,p_o,x^i,p_i,z,p_z,u)=p_oF+p_iX^i+p_z F$, defined on the manifold $T^*(\Real\times M\times\Real)\times U$. The presymplectic form is $\omega=\d x^o\wedge\d p_o+\d x^i\wedge\d p_i+\d z\wedge\d p_z$, with kernel given by the tangent vector fields to $U$, and the Hamiltonian vector field $X_H$, solution to the equation $\textbf{i}_{X_H}\omega=\d H$, is locally given by
\begin{eqnarray*}
X_H&=&F\frac{\partial}{\partial x^o}+0\frac{\partial}{\partial p_o}+
X^i\frac{\partial}{\partial x^i}+F\frac{\partial}{\partial z}
\\
&-&\left(p_o\frac{\partial F}{\partial x^i}+p_j\frac{\partial X^j}{\partial x^i}+p_z\frac{\partial F}{\partial x^i}\right)
\frac{\partial }{\partial p_i}-
\left(p_o\frac{\partial F}{\partial z}
+p_z\frac{\partial F}{\partial z}\right)
\frac{\partial }{\partial p_z}\\
&+&A^a\frac{\partial}{\partial u^a}\,\, ,
\end{eqnarray*}
where the last term corresponds to the kernel of $\omega$.

Observe that this solution exists all over the manifold
$T^*(\Real\times M\times\Real)\times U$ and that $p_o$ is constant for every curve solution to the problem. 

Being a presymplectic system, the compatibility equations are given by $\mathbf{i}(Z)\d\,H=0$ for every $Z\in\ker\omega$, that is equations

\begin{equation}\label{hpmp-controls}
\frac{\partial H}{\partial u^1}=0,\ldots,\frac{\partial H}{\partial u^k}=0
\end{equation} 

\noindent which, together with the equations coming from the vector field $X_H$, give us a set of equations to solve the optimal control problem. Recall that these compatibility conditions are the same that the optimality ones.

As in ordinary optimal control problems, suppose that the compatibility equations allow us to determine the controls $u^1,\ldots,u^k$, that is we can obtain $u^a=\psi(x^o,x^i,p_o,p_i)$, then we say that the optimal control problem is \textbf{regular}, otherwise it is called \textbf{singular}. In the singular case, it is necessary to apply an algorithm of constraints, that is to go to higher order conditions, to obtain the controls perhaps on a submanifold of $T^*(\Real\times M\times\Real)\times U$. 

The differential equations associated with the above vector field $X_H$, together with equations (\ref{hpmp-controls}) are the solution equations to the Herglotz optimal control problem.

\bigskip
\begin{remark}

To understand the significance of these equations, we can compare the above set of equations with the corresponding ones for a classical optimal control system. Apart from the compatibility conditions, which are the same, the vector field solution, see Theorem \ref{presym-pmp}, was given by

\begin{equation} X_N=F\frac{\partial}{\partial x^o}+X^i\frac{\partial}{\partial x^i}-
\left(\lambda_o\frac{\partial F}{\partial x^i}+
p_j\frac{\partial X^j}{\partial x^i}\right)\frac{\partial}{\partial p_i}\, .
\end{equation}
Comparing this vector field $X_N$ with the above $X_H$, in this last we have a new variable, $z$, hence two new terms, one for $\dot z$ and the other for $\dot p_z$. Moreover, the term corresponding to $p_i$ has changed. 

But if the cost function $F$ does not depend on $z$, then we have that $\dot p_z=0$, hence $p_z=\mathrm {constant}$, and both equations, the classical and the Herglotz optimal control, are the same.  In fact in this last case, we can change the differential equation $\dot z=F(x,z,u)$ and the optimality condition by the integral to be optimized
$$
\int_a^b F(x,u)\d t
$$
and we obtain exactly the classical problem.

Hence, as we proposed at the beginning of the section, we actually have a generalization of the classical optimal control problem from the point of view of the equations solving the problem.
\end{remark}

\subsection{Contact formulation for the normal solutions}
We can analyze the set of \textbf{normal solutions}, that is $p_o\neq 0$, in the aim of the Section \ref{normalsolut} and obtain these solutions as integral curves of contact dynamical systems.

\bigskip
To proceed suppose we are in the regular situation, that is the maximality conditions allows us to determine the controls. To study this situation we can  fix the controls, they are determined by the last equations solution to the problem, and analyze the other equations as solutions of a symplectic problem. Then, once fixed $u=u_o$,  our manifold is $T^*(\Real\times  M\times\Real)$. In this manifold we can analize the problem as a contact dynamical system.

For a given $\lambda_o\in\Real$, $\lambda_o\neq 0$, consider the submanifold $N_{\lambda_o}\subset T^*(\Real\times  M\times\Real)$, given by $p_o=\lambda_o$ and the natural injection $j_{\lambda_o}:N_{\lambda_o} \hookrightarrow T^*(\Real\times  M\times\Real)$. Let $\eta=-j_{\lambda_o}^*\theta\in\Omega^1(N_{\lambda_o})$. Then we have

\begin{lem}
 For every fixed $u\in U$, the manifold $(N_{\lambda_o}, \eta)$ is a contact manifold. Its Reeb vector field is given by 
 $$R_{\lambda_o}=-\frac{ 1}{ \lambda_o}\frac{\partial }{\partial x^o}$$
\end{lem}

The proof is straightforward using the local expression of $\eta$
$$
\eta = -\lambda_o\d x^o-p_i\d x^i-p_z\d z\,\, .
$$
 
 Let $H_{N_{\lambda_o}}=j_{\lambda_o}^*H$ and consider the Hamiltonian contact system given by $(N_{\lambda_o},\eta, H_{N_{\lambda_o}})$. Let $Z\in \vf(N_{\lambda_o})$ the corresponding Hamiltonian vector field, that is the solution to the contact equations
 
$$
 \mathbf i(Z)\eta=-H_{N_{\lambda_o}},\,\,\,\mathbf i(Z)\d\,\eta=\d\, H_{N_{\lambda_o}}-(L(R_{\lambda_o})H_{N_{\lambda_o}})\eta\, 
 $$
whose local expression is
\begin{eqnarray*}
X_H&=&F\frac{\partial}{\partial x^o}+0\frac{\partial}{\partial p_o}+
X^i\frac{\partial}{\partial x^i}+F\frac{\partial}{\partial z}\\
&-&\left(p_o\frac{\partial F}{\partial x^i}+
p_j\frac{\partial X^j}{\partial x^i}+
p_z\frac{\partial F}{\partial x^i}\right)
\frac{\partial }{\partial p_i}-
\left(p_o\frac{\partial F}{\partial z}
+
p_z\frac{\partial F}{\partial z}\right)
\frac{\partial }{\partial p_z}\,\, .
\end{eqnarray*}

\bigskip
With the above expressions and comments we have proven the 

\begin{teor}\label{thm:EOCP_precontact}
  The normal solutions to the problem \ref{hocp} corresponding to $p_o=\lambda_o\neq 0$ are the projections to $\Real\times M\times\Real\times U$ of the curves solution to the contact Hamiltonian problem given by $(N_{\lambda_o},\eta, H_{N_{\lambda_o}})$. 
\end{teor}

The corresponding differential equations for the curves solution to this Hamiltonian contact problem are :
\begin{eqnarray*}
\dot x^o  &  =  &  F \\
 \dot x^i   &  =  &  X^i\\
\dot z &=& F\\
\dot p_i   &  =  &-p_o\frac{\partial F}{\partial x^i}-
p_j\frac{\partial X^j}{\partial x^i}-p_z\frac{\partial F}{\partial x^i}\\
\dot p_z   &  =   &-
 p_o\frac{\partial F}{\partial z}
-p_z\frac{\partial F}{\partial z}
\end{eqnarray*}
Together with the maximization condition, that is the constraints obtained from the compatibility of the presymplectic equation
$$
 \frac{\partial H}{\partial u^1} = 0,\ldots\ldots,\frac{\partial H}{\partial u^k}=0 
$$

\subsection{Reduction of the problem}

We remark that this problem is a generalization of Herglotz variational principle. On the previous section, we showed that the equations obtained through the Pontryagin Maximum Principle could be reduced to obtain the Herglotz equation. In this section, we show that a similar reduction can be applied in this more general case. 

We see from the differential equations above that, taking the same initial condition for both variables, we will have $x^o = z$ for the solutions of problem~\ref{hocp}. Then one of them is irrelevant to the problem, we can eliminate it. As the momentum corresponding to $x^o$ is constant, we can eliminate the pair $(x^o, p_o)$. Observe that, in fact, $p_z$ is also irrelevant to the problem. Indeed, we can reduce the dimension of the state space of the problem; this new manifold is what we will now construct. Consider the Hamiltonian
\begin{equation}
    \begin{aligned}
        H_0 :  W_0 = T^*M \times \Real \times U &\to \mathbb{R}, \\
        (x^i,p_i,z, u^a) &\mapsto p_i X^i(x^i,z, u^a) - \pr_1^* F(x,z,u).
    \end{aligned}
\end{equation}
and the canonical contact form on $T^*M \times \Real$
\begin{equation}
    \eta_0 = \dd z - p_i \dd x^i.
\end{equation}

\begin{teor}\label{thm:hocpr}
  The normal solutions to the problem \ref{hocp} corresponding to $p_o=\lambda_o$ are the projections to $\Real\times M$ of the curves solution to the contact Hamiltonian problem given by $(T^*M \times \Real, \eta_0, H_0)$. 

The differential equations solution to this contact dynamical system are
  \begin{subequations}
    \begin{align}
        \dot{q}^i &= X^i,\\
        \dot{p}_i &= p_i \frac{\partial F}{\partial z} - p_j \frac{\partial X^j}{\partial x^i} + \frac{\partial F}{\partial x^i} - \frac{\partial X^j}{\partial z} p_i p_j, \\
        \dot{z} &= F
    \end{align}
   subjected to the constraints
    \begin{equation}
         \frac{\partial H}{\partial u^a} = \frac{\partial F}{\partial u^a} - p_j \frac{\partial X^j}{\partial u^a} = 0.
    \end{equation}
\end{subequations}

\end{teor}

\begin{remark}
In the case that the problem is singular, one would work instead with the precontact system $(T^*M \times \Real \times U, \eta_0, H_0)$, applying the  appropriate constraint algorithm.
\end{remark}

\begin{proof}
  Let $\gamma$ be a solution of the Herglotz optimal control problem. By Theorem~\ref{thm:EOCP_precontact}, we know that there exists a solution curve $\sigma$ of the corresponding contact system on $N_{\lambda_0}$. In order to prove this theorem, we will project $\sigma$ onto a solution of the system $(T^*M \times \Real, \eta_0, H_0)$.

  First of all, notice that the solutions satisfy $x_0 = z$, hence $\sigma$ will lie on the submanifold $j:\tilde{N}_{\lambda_0} \to N_{\lambda_0}$ defined by $x_0 = z$.
  
  This submanifold is tangent to the equations of motion $X$ of the precontact system $(N_{\lambda_0}, \eta_{\lambda_0}, H_{\lambda_0})$. Indeed, the restriction of $X$ to $\tilde{N}_{\lambda_0}$ are just the equations of motion $\tilde{X}$ of the induced precontact system $(\tilde{N}_{\lambda_0}, \tilde{\eta}_{\lambda_0} = j^* \eta_{\lambda_0}, \tilde{H}_{\lambda_0} = j^* H_{\lambda_0})$. In coordinates
  \begin{subequations}
    \begin{align}
      \tilde{\eta}_{\lambda_0} &= (- \lambda_0 - p_z) \dd z - p_i \dd x^i,\\
      \tilde{H}_{\lambda_0} &= (\lambda_0 + p_z) F + p_i X^i
    \end{align}
  \end{subequations}

  Consider the following commutative diagram,
  \begin{equation}
\begin{tikzcd}
  & \tilde{N}_{\lambda_0} \arrow[d, "\tau"] \arrow[ld, "\Phi_{\lambda_0}"'] \\
W_0 \arrow[rd, "\tau_0"] & \Real \times M \times \Real \arrow[d] \arrow[d, "\pi_1"]    \\
  & M \times \Real                                             
\end{tikzcd}
\end{equation}
  where
  \begin{equation}
    \Phi_{\lambda_0}(x^i,z,p_i,p_z) = (x^i, -(\lambda_o+p_z)p_i, z).
  \end{equation}

  Notice that $\Phi_{\lambda_0}$ is a submersion and a conformal equivalence of precontact systems:
  \begin{subequations}
    \begin{align}
      \Phi_{\lambda_0}^* \eta_0 & = -(\lambda_o+p_z) \tilde{\eta}_{\lambda_0}, \\
      \Phi_{\lambda_0}^* H_0    & = -(\lambda_o+p_z) \tilde{H}_{\lambda_0},
    \end{align}
  \end{subequations}
By Theorem~\ref{thm:precontact_equivalence} projections of the solution curves of the precontact system on $\tilde{N}_{\lambda_0}$ are solution curves to the contact system on $TQ\times \Real$. 
\end{proof}

As a consequence of this theorem, we can obtain again the Herglotz equations. Consider the Herglotz problem in section~\ref{herglotz_problem} for a Lagrangian $L:TQ \times \Real \to \Real$. Notice that this problem is a particular case of the Herglotz optimal control problem, where 
\begin{itemize}
    \item Controls are the velocities $u^a = v^i$.
    \item The cost function is the Lagrangian $F=L$.
    \item The control equation is $X =v^i \frac{\partial }{\partial x^i}$.
\end{itemize}

The solutions to this problem are given by Theorem~\ref{thm:hocpr}:
\begin{align}
        \dot{q}^i &= v^i,\\
        \dot{p}_i &= p_i \frac{\partial L}{\partial z}  + \frac{\partial L}{\partial q^i} \\
        \dot{z} &= L\\
        \intertext{with the constraints}
        \frac{\partial L}{\partial v^i} &= p_i,
\end{align}
which are precisely Herglotz equations.

\section{Application: Optimal control on thermodynamic systems}
One possible application of this theory is the study of thermodynamic processes which minimize or maximize some thermodynamic potential. As an example, we apply our formalism to the control systems considered in~\cite{vds-2018}.

The relation between symplectic and contact manifolds via the symplectification procedure has permitted to go deeper in the geometric description of thermodinamic systems. This way has been explored in \cite{BaVa-2001} (see also \cite{arn78,LiMa-1987,IbLeMaDi-1997}).

\subsection{Homogeneous Hamiltonian systems and contact systems}\label{sec:homogenization}
There is a close relationship between homogeneous symplectic and contact systems, see for example~\cite{vds-2018} where this relation is studied. Here we briefly recall the ideas we need to follow the example.

In the general case, if $\pi:M\to B$ is a vector bundle, a function $F:M\to\Real$ is \textit{homogeneous} if, for any $e_p\in M_p=\pi^{-1}(p)$ with $\pi(e_p)=p\in B$, we have $F(\lambda e_p)=\lambda F(e_p)$. In this situation the function $F$ can be projected to the projective bundle $\mathcal{P}(M)$ over $B$ obtained by projectivization on every fibre. We are interested in the case that $M=T^*(Q\times\Real)\to Q\times\Real$, with natural coordinates $(q^i,z,P_i,P_z)$

Let $H$ be an homogeneous Hamiltonian function on $T^*( Q \times \Real)$. Locally, we have that $H(q^i,z, \lambda P_i,  \lambda P_z) = \lambda H(q^i,z,P_i,P_z)$, for all $\lambda \in \Real$. Equivalently, one can write
\begin{equation}
  H(q^i,z,P_i,P_z) = -P_z\,  h(q^i,-P_i/P_z, z),
\end{equation}
for $P_z\neq 0$, where $h: T^*Q \times \Real \to \Real$, $h(q^i,z,p_i)=H(q^i,z,-p_i,-1)$ is well defined.

With the above changes, we have identified the manifold $T^*Q \times \Real$ as the projective bundle $\mathcal{P}(T^* (Q \times \Real))$ of the cotangent bundle $T^*(Q \times \Real)$ taking out the points at infinity,  that is the subset defined by $\{P_z = 0\}$.

Following \cite[Section~4.1]{vds-2018}, the map
\begin{equation}\label{dehomogeneization}
  \begin{aligned}
    \Phi: T^*( Q \times \Real) \setminus \{p_z = 0 \} &\to T^*Q \times \Real\\
    (q^i,z,P_i,P_z) &\to (q^i,P_i/P_z, z) = (q^i, p_i ,z),
  \end{aligned}
\end{equation}
sends the Hamiltonian symplectic system $(T^*( Q \times \Real)\setminus \{p_z = 0 \} , \omega_{Q \times \Real}, H)$ onto the Hamiltonian contact system $(T^*Q \times \Real, \eta_{Q}, h)$, where $\omega_{Q \times \Real} = \dd q^i \wedge \dd P_i + \dd z \wedge \dd P_z$ and $\eta_Q = \dd z - p_i \dd q^i$. Observe that the natural coordinates of $T^*Q\times\Real$, denoted by $(q^i,p_i,z)$, correspond to the homogeneous coordinates in the projective bundle.

In fact, the map $\Phi$ is the projectivization; i.e., the map that sends each point in the fibers of $T^* (Q \times \Real)$ to the line that passes through it and the origin. 

It can be shown that $\Phi$ provides a bijection between conformal contactomorphisms and homogeneous symplectomorphisms. Moreover, $\Phi$ maps homogeneous Lagrangian submanifolds $\mathcal{L} \subseteq T^*( Q \times \Real)$ onto Legendrian submanifolds $\mathbb{L} = \phi(\mathcal{L}) \subseteq T^*Q \times \Real$. See \cite{vds-2018} and Section \ref{apthersys} for more details on this topics.

\subsection{Control of contact systems}\label{cocontsys}
 
On the contact natural manifold $T^* Q\times\Real$, with coordinates $(q^i,p_i,z)$, assume that we are given a parametrized family of Hamiltonians $h: T^* Q\times\Real \times U \to \Real$, $U\subset \Real^k$, with Hamiltonian contact vector fields $X_{h_u}$, where $h_u(q^i,z,p_i)=h(q^i,z,p_i,u)$. Then we can define the control system $Z(q,p,z,u) = X_{h_u}(q,p,z)$, where the following diagram is commutative:
$$
\xymatrix{
&T(T^* Q\times\Real)\ar[d]^{\tau_{T^* Q\times\Real}}\\
{T^* Q\times\Real\times U}\ar[ur]^{Z}\ar[r]^{\quad\pi}&T^* Q\times\Real\\
I\ar[u]^\gamma\ar[ru]^{\Gamma}
}
$$

A curve $\gamma:I\to T^*Q\times\Real\times U$ is an integral curve of $Z$, that is $\Gamma'=Z\circ\gamma$, if in local coordinates satisfies the differential equations
\begin{align*}
  \frac{\dd q^i}{\dd t} & =  \frac{\partial h_u}{\partial p_i}, \\
  \frac{\dd p_i}{\dd t} & =  - \frac{\partial h_u}{\partial q^i} - p_i \frac{\partial h_u}{\partial z},\\
  \frac{\dd z}{\dd t} & =  p_i \frac{\partial h_u}{\partial p_i} - h_u.
\end{align*} 

One can consider the Herglotz optimal control problem given by $Z$, as we stated in  Section~\ref{sec:hocp_vf}. Then, by Theorem~\ref{thm:hocpr}, we know that the normal solutions are the projections of the solutions to the contact system $(T^*(T^*Q) \times \Real, \eta_{T^*Q }, H)$, where
\begin{equation}\label{eq:Hamiltonian_contact_system_control}
  H =  p_{q^i}\frac{\partial h_u}{\partial p_i} - p_{p_i}\frac{\partial h_u}{\partial q^i} - p_i \frac{\partial h_u}{\partial z} -  p_i \frac{\partial h_u}{\partial p_i} + h_u.
\end{equation}

\subsection{Application to thermodynamic systems}\label{apthersys}

We consider thermodynamic systems in the so called \emph{entropy representation}. Hence
the \emph{thermodynamic phase space}, representing the extensive variables,  is the manifold $T^* Q  \times \Real$, equipped with its canonical contact form
\begin{equation}
  \eta_Q = \dd S - p^i \dd q^i.
\end{equation}
The local coordinates on the configuration manifold $Q$ are $(q^i,S)$, where $S$ is the total entropy and $q^i$'s denote the rest of extensive variables.
Other variables, such as the internal energy, may be chosen instead of the entropy, by means of a Legendre transformation.

The state of a thermodynamic system always lies on the equilibrium submanifold $\mathbb{L}\subseteq T^* Q  \times \Real$, which is a Legendrian submanifold, that is, $\eta \vert_{T\mathbb{L}} = 0$ and $\dim\mathbb{L}=\dim Q=n$. The pair $(T^* Q  \times \Real, \mathbb{L})$ is a \emph{thermodynamic system}. The equations (locally) defining $\mathbb{L}$ are called the \emph{state equations} of the system.

On a thermodynamic system $(T^* Q  \times \Real, \mathbb{L})$, one can consider the dynamics generated by a Hamiltonian vector field $X_h$ associated to a Hamiltonian $h$. If this dynamics represents \emph{quasistatic processes}, meaning that at every time the system is in equilibrium, that is, its evolution states remain in the submanifold $\mathbb{L}$, it is required for the contact Hamiltonian vector field $X_h$ to be tangent to $\mathbb{L}$. This happens if and only if $h$ vanishes on $\mathbb{L}$.

Equivalently, by section~\ref{sec:homogenization}, one can consider the \emph{extended thermodynamic phase space} $T^* (Q \times \Real)$ with its canonical symplectic form
\begin{equation}
  \omega_{Q\times\Real} = 
   \dd q^i \wedge \dd P_i + \dd S \wedge \dd P_S.
\end{equation}
In this formulation, a thermodynamic system is a tuple $(T^* (Q  \times \Real), \mathcal{L}))$, where $\mathcal{L}$ is a homogeneous Lagrangian submanifold. Dynamics are given by a homogeneous Hamiltonian $K$. See \cite{vds-2018} for details and recall we have identified, in Section \ref{sec:homogenization}, the bundle $T^*Q\times\Real$ with the projective bundle $\mathcal{P}(T^* (Q \times \Real))$.

Port-thermodynamic systems were introduced in~\cite{vds-2018}, but in a homogeneous symplectic formalism. 
\begin{definition}[Port-thermodynamic system]\label{portthermo}
   A \emph{port-thermodynamic system} on $T^*(Q \times \Real)$ is defined as a pair $(\mathcal{L},K)$, where the homogeneous Lagrangian submanifold $\mathcal{L} \subset T^*(Q\times \Real)$ specifies the state properties. The dynamics is given by the homogeneous Hamiltonian dynamics with parametrized homogeneous Hamiltonian $K:= K^a + {K^c}_a u^a : T^*(Q\times \Real) \to \Real, \, u \in \Real^k$, $K^c: T^*(Q\times \Real) \to \Real^k$, with $K^a$, $K^c$ both equal to zero on the points of $\mathcal{L}$, and $K^a$ as the internal Hamiltonian. One need the additional condition
  \begin{equation}
  \frac{\partial K}{\partial S} |_{\mathcal{L} }\geq 0,
  \end{equation}
  so that the second law of thermodynamics holds.
  \end{definition}

    Using the results of section~\ref{sec:homogenization}, we could instead consider the following contact formulation.
  \begin{definition}[Port-thermodynamic system, contact formalism]\label{portthermocontact}
    A port-thermodynamic system on $(T^*Q \times \Real, \eta_Q)$ is defined as a pair $(\mathbb{L},h)$, where the Legendrian submanifold ${\mathbb{L}} \subset T^*Q \times \Real$ specifies the state properties. The dynamics is given by the contact Hamiltonian dynamics with parametrized contact Hamiltonian $h = h^a + h^c_a u^a : T^*Q \times \Real \to \Real, \, u \in \Real^m$, $h^c: T^*Q\times \Real \to \Real^k$, with $h^a,h^c$ zero on $\mathbb{L}$, and the internal Hamiltonian $h^a$ satisfying 
    \begin{equation}
    \frac{\partial h}{\partial S} |_{\mathbb{L} }\geq 0,
    \end{equation}
      so that the second law of thermodynamics holds.
 \end{definition}
 
 Our theory provides tools to understand which of the available thermodynamic processes minimize the entropy production of the system. Observe that we can consider processes that maximize or minimize other thermodynamic variables, such as the energy, via a Legendre transform. 
 \subsection{Example: Gas-Piston-Damper system}
 We end this section with an explicit example which can be found in~\cite{vds-2018}.
 
  Consider an adiabatically isolated cylinder closed by a piston containing a gas with internal energy $U(V,S)$.
  
  The extended phase space has the following extensive variables
 \begin{itemize}
     \item the momentum of the piston $\pi$,
     \item the volume of the gas $V$,
     \item the energy $E$,
     \item the entropy $S$.
 \end{itemize}
 They correspond to $Q\times\Real$ with local coordinates $(V,\pi,E,S)$.
 The Legendrian submanifold is given by
 \begin{equation}
       \mathbb{L} = \{(V,\pi,E,p_V,p_{\pi},p_E,S) | E= \frac{\pi^2}{2m} + U(S,V),   p_V = -p_E \frac{\partial{U}}{\partial{V}} , p_{\pi}= -p_E \frac{\pi}{m}, p_E = 1/\frac{\partial U}{\partial S} \}
 \end{equation}
  
  The energy is then given by 
  \begin{equation}
        h = p_V\frac{\pi}{m} +p_{\pi}\left(-\frac{\partial U}{\partial V} -d\frac{\pi}{m}\right) - \frac{d (\frac{\pi}{m})^2}{\frac{\partial U}{\partial S}} + \left(p_{\pi} + p_E \frac{\pi}{m}\right)u,
  \end{equation}
where $d$ is the diameter of the piston and $m$ is its mass.

The Hamiltonian vector field is given by
\begin{equation}
\begin{gathered}
    X_h = \frac{{\pi}}{m} \frac{\partial}{\partial V } + \left( -\frac{{\pi} d}{m} + u - \frac{\partial\,U}{\partial V} \right) \frac{\partial}{\partial {\pi} }
    + \frac{{\pi} u}{m} \frac{\partial}{\partial E }\\
    + \left( {\left({p_\pi} \frac{\partial^2\,U}{\partial V\partial S} - \frac{{\pi}^{2} d \frac{\partial^2\,U}{\partial S ^ 2}}{m^{2} \left(\frac{\partial U}{\partial S}\right)^{2}}\right)} {p_V} + {p_\pi} \frac{\partial^2\,U}{\partial V ^ 2} - \frac{{\pi}^{2} d \frac{\partial^2\,U}{\partial V\partial S}}{m^{2} \left(\frac{\partial U}{\partial S}\right)^{2}} \right) \frac{\partial}{\partial {p_V} }\\
    + \left( {\left({p_\pi} \frac{\partial^2\,U}{\partial V\partial S} - \frac{{\pi}^{2} d \frac{\partial^2\,U}{\partial S ^ 2}}{m^{2} \left(\frac{\partial U}{\partial S}\right)^{2}}\right)} {p_\pi} + \frac{d {p_\pi}}{m} - \frac{{p_E} u}{m} - \frac{{p_V}}{m} + \frac{2 \, {\pi} d}{m^{2} \frac{\partial U}{\partial S}} \right) \frac{\partial}{\partial {p_\pi} }\\ + {\left({p_\pi} \frac{\partial^2\,U}{\partial V\partial S} - \frac{{\pi}^{2} d \frac{\partial^2\,U}{\partial S ^ 2}}{m^{2} \left(\frac{\partial U}{\partial S}\right)^{2}}\right)} {p_E} \frac{\partial}{\partial {p_E} } \\
    + \left( \frac{{\pi}^{2} d}{m^{2} \frac{\partial U}{\partial S}} \right) \frac{\partial}{\partial S }
\end{gathered}
\end{equation}

We construct the contact Hamiltonian system $(T^*(T^*Q)\times \Real, \eta_{T^*Q}, H)$ as in~\eqref{eq:Hamiltonian_contact_system_control}:
\begin{equation}
\begin{gathered}
  H = 
-{\left(\frac{d {p_\pi}}{m} - \frac{{p_E} u}{m} - \frac{{p_V}}{m} + \frac{2 \, {\pi} d}{m^{2} \frac{\partial U}{\partial S}}\right)} {P_\pi} \\- {\left({p_\pi} \frac{\partial^{2}}{(\partial V)^{2}}U\left(V, S\right) - \frac{{\pi}^{2} d \frac{\partial^{2}}{\partial V\partial S}U\left(V, S\right)}{m^{2} \frac{\partial U}{\partial S}^{2}}\right)} {P_V}\\ - {\left(\frac{{\pi} d}{m} - u + \frac{\partial}{\partial V}U\left(V, S\right)\right)} {P_{p_\pi}} + \frac{{\pi} {P_{p_E}} u}{m} + \frac{{\pi} {P_{p_V}}}{m} - \frac{{\pi}^{2} d}{m^{2} \frac{\partial U}{\partial S}},
\end{gathered}
 \end{equation}
 where we denote by $q^i, p_{q^i}, \Pi_{q^i}, \Pi_{p_{q^i}}$ the natural coordinates on $T^*T^*Q$, where $q^i$ runs through $V,\pi,E$, and $\Pi_{q^i}, \Pi_{p_{q^i}}$ are the corresponding moments to $q^i,p_i$ respectively.
 
The solutions to the control problem are then the integral curves of the Hamiltonian vector field of this system, which are the following

\begin{equation}
\begin{aligned}
{\dot{V}} &= \frac{{\pi}}{m} \\
{\dot{ {\pi}}} &= -\frac{{\pi} d}{m} + u - \frac{\partial\,U}{\partial V} \\
{\dot{ E}} &= \frac{{\pi} u}{m} \\ 
{\dot{ {p_V}}} &= {\left({p_\pi} \frac{\partial^2\,U}{\partial V\partial S} - \frac{{\pi}^{2} d \frac{\partial^2\,U}{\partial S ^ 2}}{m^{2} \left(\frac{\partial U}{\partial S}\right)^{2}}\right)} {p_V} + {p_\pi} \frac{\partial^2\,U}{\partial V ^ 2} - \frac{{\pi}^{2} d \frac{\partial^2\,U}{\partial V\partial S}}{m^{2} \left(\frac{\partial U}{\partial S}\right)^{2}} \\
\dot{p_\pi} &= {\left({p_\pi} \frac{\partial^2\,U}{\partial V\partial S} - \frac{{\pi}^{2} d \frac{\partial^2\,U}{\partial S ^ 2}}{m^{2} \left(\frac{\partial U}{\partial S}\right)^{2}}\right)} {p_\pi} + \frac{d {p_\pi}}{m} - \frac{{p_E} u}{m} - \frac{{p_V}}{m} + \frac{2 \, {\pi} d}{m^{2} \frac{\partial U}{\partial S}} \\
\dot{p_E} &= {\left({p_\pi} \frac{\partial^2\,U}{\partial V\partial S} - \frac{{\pi}^{2} d \frac{\partial^2\,U}{\partial S ^ 2}}{m^{2} \left(\frac{\partial U}{\partial S}\right)^{2}}\right)} {p_E} \\
\dot{ S} &=\frac{{\pi}^{2} d}{m^{2} \frac{\partial U}{\partial S}}\\
\dot{\Pi_V} &=  \alpha {\Pi_V} - \frac{{\Pi_\pi}}{m}\\
\dot{\Pi_\pi} &=  \alpha {\Pi_E} - \frac{{\Pi_\pi} u}{m}\\
\dot{\Pi_{p_V}} &= -{p_\pi} {p_V} {\Pi_V} \frac{\partial^3\,U}{\partial V ^ 2\partial S} - {p_\pi} {\Pi_V} \frac{\partial^3\,U}{\partial V ^ 3} - {p_V} {\Pi_{p_\pi}} \frac{\partial^2\,U}{\partial V\partial S} - {p_\pi} {\Pi_{p_V}} \frac{\partial^2\,U}{\partial V\partial S} \\& +  \alpha {\Pi_{p_V}} - {\Pi_{p_\pi}} \frac{\partial^2\,U}{\partial V ^ 2} + \frac{{\pi}^{2} d {p_V} {\Pi_V} \frac{\partial^3\,U}{\partial V\partial S ^ 2}}{m^{2} \left(\frac{\partial U}{\partial S}\right)^{2}} - \frac{2 \, {\pi}^{2} d {p_V} {\Pi_V} \frac{\partial^2\,U}{\partial V\partial S} \frac{\partial^2\,U}{\partial S ^ 2}}{m^{2} \left(\frac{\partial U}{\partial S}\right)^{3}} \\& - \frac{2 \, {\pi}^{2} d {\Pi_V} \frac{\partial^2\,U}{\partial V\partial S}^{2}}{m^{2} \left(\frac{\partial U}{\partial S}\right)^{3}} + \frac{{\pi}^{2} d {\Pi_V} \frac{\partial^3\,U}{\partial V ^ 2\partial S}}{m^{2} \left(\frac{\partial U}{\partial S}\right)^{2}} + \frac{2 \, {\pi} d {p_V} {\Pi_\pi} \frac{\partial^2\,U}{\partial S ^ 2}}{m^{2} \left(\frac{\partial U}{\partial S}\right)^{2}} + \frac{{\pi}^{2} d {\Pi_{p_V}} \frac{\partial^2\,U}{\partial S ^ 2}}{m^{2} \left(\frac{\partial U}{\partial S}\right)^{2}} + \frac{2 \, {\pi} d {\Pi_\pi} \frac{\partial^2\,U}{\partial V\partial S}}{m^{2} \left(\frac{\partial U}{\partial S}\right)^{2}}\\
\dot{\Pi_{p_\pi}} &= -{p_\pi}^{2} {\Pi_V} \frac{\partial^3\,U}{\partial V ^ 2\partial S} - 2 \, {p_\pi} {\Pi_{p_\pi}} \frac{\partial^2\,U}{\partial V\partial S} +  \alpha {\Pi_{p_\pi}} + \frac{{\pi}^{2} d {p_\pi} {\Pi_V} \frac{\partial^3\,U}{\partial V\partial S ^ 2}}{m^{2} \left(\frac{\partial U}{\partial S}\right)^{2}} - \frac{2 \, {\pi}^{2} d {p_\pi} {\Pi_V} \frac{\partial^2\,U}{\partial V\partial S} \frac{\partial^2\,U}{\partial S ^ 2}}{m^{2} \left(\frac{\partial U}{\partial S}\right)^{3}}  \\ & - \frac{d {\Pi_{p_\pi}}}{m} + \frac{{\Pi_{p_E}} u}{m} + \frac{2 \, {\pi} d {p_\pi} {\Pi_\pi} \frac{\partial^2\,U}{\partial S ^ 2}}{m^{2} \left(\frac{\partial U}{\partial S}\right)^{2}} + \frac{{\pi}^{2} d {\Pi_{p_\pi}} \frac{\partial^2\,U}{\partial S ^ 2}}{m^{2} \left(\frac{\partial U}{\partial S}\right)^{2}} + \frac{{\Pi_{p_V}}}{m} + \frac{2 \, {\pi} d {\Pi_V} \frac{\partial^2\,U}{\partial V\partial S}}{m^{2} \left(\frac{\partial U}{\partial S}\right)^{2}} - \frac{2 \, d {\Pi_\pi}}{m^{2} \frac{\partial U}{\partial S}}\\
\dot{\Pi_{p_E}} &= -{p_E} {p_\pi} {\Pi_V} \frac{\partial^3\,U}{\partial V ^ 2\partial S} - {p_\pi} {\Pi_{p_E}} \frac{\partial^2\,U}{\partial V\partial S} - {p_E} {\Pi_{p_\pi}} \frac{\partial^2\,U}{\partial V\partial S} +  \alpha {\Pi_{p_E}} + \frac{{\pi}^{2} d {p_E} {\Pi_V} \frac{\partial^3\,U}{\partial V\partial S ^ 2}}{m^{2} \left(\frac{\partial U}{\partial S}\right)^{2}} \\
 &- \frac{2 \, {\pi}^{2} d {p_E} {\Pi_V} \frac{\partial^2\,U}{\partial V\partial S} \frac{\partial^2\,U}{\partial S ^ 2}}{m^{2} \left(\frac{\partial U}{\partial S}\right)^{3}} + \frac{2 \, {\pi} d {p_E} {\Pi_\pi} \frac{\partial^2\,U}{\partial S ^ 2}}{m^{2} \left(\frac{\partial U}{\partial S}\right)^{2}} + \frac{{\pi}^{2} d {\Pi_{p_E}} \frac{\partial^2\,U}{\partial S ^ 2}}{m^{2} \left(\frac{\partial U}{\partial S}\right)^{2}},
\end{aligned}
\end{equation}
where
\begin{align*}
    \alpha &= \frac{\partial F}{\partial S}-\Pi_j \frac{\partial X_j} {\partial S} \\ 
    &= \scriptstyle -{p_E} {p_\pi} {\Pi_{p_E}} \frac{\partial^3\,U}{\partial V\partial S ^ 2} - {p_\pi}^{2} {\Pi_{p_\pi}} \frac{\partial^3\,U}{\partial V\partial S ^ 2} - {p_\pi} {p_V} {\Pi_{p_V}} \frac{\partial^3\,U}{\partial V\partial S ^ 2} - {p_\pi} {\Pi_{p_V}} \frac{\partial^3\,U}{\partial V ^ 2\partial S} + {\Pi_\pi} \frac{\partial^2\,U}{\partial V\partial S} - \frac{2 \, {\pi}^{2} d {p_E} {\Pi_{p_E}} \left(\frac{\partial^2\,U}{\partial S ^ 2}\right)^{2}}{m^{2} \left(\frac{\partial U}{\partial S}\right)^{3}} \\ & -  \scriptstyle \frac{2 \, {\pi}^{2} d {p_\pi} {\Pi_{p_\pi}} \left(\frac{\partial^2\,U}{\partial S ^ 2}\right)^{2}}{m^{2} \left(\frac{\partial U}{\partial S}\right)^{3}} - \frac{2 \, {\pi}^{2} d {p_V} {\Pi_{p_V}} \left(\frac{\partial^2\,U}{\partial S ^ 2}\right)^{2}}{m^{2} \left(\frac{\partial U}{\partial S}\right)^{3}} + \frac{{\pi}^{2} d {p_E} {\Pi_{p_E}} \frac{\partial^3\,U}{\partial S ^ 3}}{m^{2} \left(\frac{\partial U}{\partial S}\right)^{2}} + \frac{{\pi}^{2} d {p_\pi} {\Pi_{p_\pi}} \frac{\partial^3\,U}{\partial S ^ 3}}{m^{2} \left(\frac{\partial U}{\partial S}\right)^{2}} \\ & \scriptstyle + \frac{{\pi}^{2} d {p_V} {\Pi_{p_V}} \frac{\partial^3\,U}{\partial S ^ 3}}{m^{2} \left(\frac{\partial U}{\partial S}\right)^{2}} + \frac{{\pi}^{2} d {\Pi_{p_V}} \frac{\partial^3\,U}{\partial V\partial S ^ 2}}{m^{2} \left(\frac{\partial U}{\partial S}\right)^{2}} - \frac{2 \, {\pi}^{2} d {\Pi_{p_V}} \frac{\partial^2\,U}{\partial V\partial S} \frac{\partial^2\,U}{\partial S ^ 2}}{m^{2} \left(\frac{\partial U}{\partial S}\right)^{3}} - \frac{ {\pi}^{2} d \frac{\partial^2\,U}{\partial S ^ 2}}{m^{2} \left(\frac{\partial U}{\partial S}\right)^{2}} + \frac{2 \, {\pi} d {\Pi_{p_\pi}} \frac{\partial^2\,U}{\partial S ^ 2}}{m^{2} \left(\frac{\partial U}{\partial S}\right)^{2}}\, ,
\end{align*}
and they are subject to the constraint
\begin{equation}
    \frac{{p_E} {\Pi_\pi}}{m} + \frac{{\pi} {\Pi_{p_E}}}{m} + {\Pi_{p_\pi}} = 0.
\end{equation}

\section{Conclusions and future work}

We have discussed several presentations of the so-called Optimal Control Theory, using presymplectic and contact geometry. These relations allows us to obtain directly a new proof of the equations solving the Herglotz variational principle. One of the main results is just the derivation of a Pontryagin Maximum Principle in the setting of Herglotz optimal control problems, a generalization of the classical optimal control. We have also exhibited how the theory can be applied to thermodynamic systems.

The results obtained in the present paper open many ways to follow, and our intention is to go in these directions; here, there are some of them:
\begin{enumerate}
\item Relations between the contact vakonomic dynamics and the Herglotz Optimal Control Problem, following the same lines that in \cite{MaCoLe-2000} and \cite{MaCoLe-2001} for the symplectic case. 

\item To study the more general  case of Herglotz variational calculus with constraints as in \cite{GraMaMu-2003} and references therein.

\item Reduction of the Herglotz Optimal Control Problem when we are in presence of symmetries, and reconstruction of the original solutions from the reduced ones (see \cite{EchMaMuRo-2003} and \cite{deLeoncomama-2004} for the classical setting).

\item Potential extensions to control problems with dissipation on Lie groupoids and algebroids, and numerical methods to solve them, (see \cite{CoMaMaMa-2006}).

\item Study of contact mechanical systems with controls, their stabilization and tracking problems (see for example, 
\cite{CoMa-2002,MuYa-2002,CoMa-2003}).
\end{enumerate}

\section*{Acknowledgements}

\hspace{5mm} M. de Le\'on and M. Lainz acknowledge the partial finantial support from MINECO Grants MTM2016- 76-072-P and the ICMAT Severo Ochoa project SEV-2015-0554.  
M. Lainz wishes to thank MICINN and ICMAT for a FPI-Severo Ochoa predoctoral contract PRE2018-083203.
M.C. Mu\~noz-Lecanda acknowledges the financial support from the 
Spanish Ministerio de Ciencia, Innovaci\'on y Universidades project
PGC2018-098265-B-C33
and the Secretary of University and Research of the Ministry of Business and Knowledge of
the Catalan Government project
2017--SGR--932.

\addcontentsline{toc}{section}{References}
\itemsep 0pt plus 1pt
\small

\bibliographystyle{abbrv}

\end{document}